\documentclass[a4paper,11pt]{article}

\usepackage{amsmath}
\usepackage{amsthm}
\usepackage{amssymb}
\usepackage{array}
\newtheorem{thmn}{Theorem}%

\newtheorem{cor}{Corollary}[section]
\newtheorem{corn}[thmn]{Corollary}
\newtheorem{lem}[cor]{Lemma}

\newtheorem{prop}[cor]{Proposition}

\theoremstyle{definition}

\numberwithin{equation}{section}

\newcommand{\R}{\mathbb R}
\newcommand{\N}{\mathbb N}
\newcommand{\B}{\mathcal B}
\newcommand{\Z}{\mathbb Z}
\newcommand{\D}{\mathcal D}

\newcommand{\bigint}{\begin{picture}(10,10)
\put(-1,2){\line(1,0){10}}
\end{picture}\kern-14pt\int}
\title{Boundary values of harmonic gradients and differentiability of Zygmund and Weierstrass functions}%
\author{JUAN J. DONAIRE , JOS\'{E} G. LLORENTE , ARTUR NICOLAU }
\date{}
\begin{document}
%\footnotetext{Partially supported by a grant of ........}
%\footnotetext{...}%\keywords{}%
%\commby{}%
% ----------------------------------------------------------------
% ----------------------------------------------------------------
\maketitle
\begin{abstract}
We study differentiability properties of Zygmund functions and
series of Weierstrass type in higher dimensions. While such
functions may be nowhere differentiable, we show that, under
appropriate assumptions, the set of points where the incremental
quotients are bounded has maximal Hausdorff dimension.
\end{abstract}
\section{Introduction and main results}

It was a widespread opinion among most of the mathematicians of the
nineteenth century that a continuous function should be
differentiable on a substantial set of points. For that reason, the
first constructions of continuous nowhere differentiable functions
in the real line - which go back to the end of the XIXth. century-
were not accepted without reservations. However, the existence of
such pathological functions was a crucial breakthrough not only in
the foundation of modern Function Theory but also in the future
development of Probability Theory and Theoretical Physics. The first
example of a continuous nowhere differentiable function is probably
due to B. Bolzano (1830), who used a geometrical construction.
However by the time Bolzano's example was published (1930)
Weierstrass had already presented his construction in the Royal
Academy of Berlin (1872, published in 1875).  Some years in between,
Cell\'{e}rier (1860, published in 1890) gave the first example by
using a trigonometric series
$$
C(x) = \sum_{n=1}^\infty a^{-n} \sin (a^n x )
$$
where $a > 0$ is a sufficiently large number.  Given $b>1$ and $0<
\alpha \leq 1$, Weierstrass proved that the continuous function
% $$
% W(x) = \sum_{n=0}^{\infty} a^{n} \cos ( \pi b^n x)
% $$
$$
f_{b, \alpha }(x) = \sum_{n=0}^\infty b^{-n \alpha } \cos ( 2 \pi
b^n x)
$$
is nowhere differentiable provided that $b \geq 7$ is an odd integer
and $\displaystyle 0 < \alpha < 1 - \frac{\log (1 + 3 \pi /2)}{\log
b}$. Hardy ([9],1916) proved that last restriction in Weierstrass'
result was superfluous in the sense that $f_{b, \alpha }$ is nowhere
differentiable as soon as $b>1$ and $0< \alpha \leq 1$. This is best
possible and the extreme case $\alpha = 1$ is the most delicate one.
% only the
% conditions $0< a < 1$ and $ab\geq 1$ are needed to prove nowhere
% differentiability, the borderline case $ab=1$ being the most
% delicate one.

During the XXth. century, a number of different geometrical and
analytic constructions  of  continuous nowhere differentiable
functions have been obtained. See [17] for a historical survey on
the subject.

For $0< \alpha \leq 1$, let $Lip_{\alpha}(\R^d )$ be the H\"{o}lder
class of bounded functions $f: \R^d \to \R$ for which there exists a
constant $C=C(f)>0$ such that $|f(x) - f(y)| < C |x-y|^\alpha$ for
any $x , y \in \R^d$.
% The question of the regularity of Weierstrass
% functions shows up when Hardy's result is rephrased in the form: "
% If $0 < \alpha \leq 1$ and $b>1$ then
% $$
% f_{b, \alpha }(x) = \sum_{n=0}^\infty b^{-n \alpha } \cos (b^n x)
% $$
% is continuous and nowhere differentiable on $\R$".
An standard trick in the theory of lacunary series gives that $f_{b,
\alpha} \in Lip_{\alpha}(\R )$ if $0< \alpha < 1$. On the other
hand, a classical theorem of Rademacher (1919) says that any
function in the Lipschitz class $Lip_{1}(\R^d )$ is differentiable
at almost every point. This implies in particular that $f_{b,1}$
cannot be locally Lipschitz on any interval. However, $f_{b,1}$
belongs to the so called \textbf{Zygmund class}. A bounded
continuous function  $f: \R^d \to \R$ is in the Zygmund class
$\Lambda_{*}(\R^d )$  if
$$
\sup _{x,h}\frac{|f(x+h) +f(x-h) -2f(x)|}{|h|} = ||f||_{*} < \infty
$$
where the supremum is taken over all $x\in \R^d$ and all $h\in \R^d
\setminus \{ 0 \}$. The Zygmund class is intermediate between the
H\"{o}lder classes in the sense that $ Lip_1 (\R^d ) \subset
\Lambda_{*}(\R^d )  \subset Lip_{\alpha }(\R^d ) $, for any $0 <
\alpha < 1$. We also introduce the \textbf{small Zygmund class}
$\lambda_{*}(\R^d )$ consisting of all bounded continuous functions
$f: \R^d \to \R$ such that
$$
\sup_{x}\frac{|f(x+h)+f(x-h)-2f(x)|}{|h|} \to 0 \, \, \, \, \,
\text{as} \, \, |h| \to 0
$$

Apart from the relations to Weierstrass functions, the Zygmund class
is  a convenient substitute of the Lipschitz class in some problems
in Harmonic Analysis. See [1],[11], [12],[14], [16] for connections
between Zygmund classes, Probability and other areas of Analysis.

The main purpose of this paper is to discuss the behavior of the
incremental quotients of a certain natural class of functions in the
Zygmund class. We will need some notation. Given a function $f: \R^d
\to \R$, define the sets
\begin{align}
\D (f) = & \{ x\in \R^d \, : \, \limsup_{|h| \to
0}\frac{|f(x+h)-f(x)|}{|h|} < \infty \} \, , \\
\D_0 (f) = & \{ x\in \R^d \, : \, \, f \, \text{is differentiable
at} \, x \, \}
\end{align}
Points in $\D (f)$ are sometimes called \emph{slow points} of $f$.
In [18], it was pointed out that if $f\in \Lambda_{*}(\R )$ (resp.
$f\in \lambda_{*}(\R )$) then $\D (f)$ (resp. $\D_0 (f)$ ) must be
dense on any interval. On the other hand, from the classical theory
of lacunary trigonometric series (see [19]) we have $m_1 (\D
(f_{b,1}))= 0$ where, hereafter, $m_{d}$ denotes $d$-dimensional
Lebesgue measure. From a metrical point of view, the definitive
answer in dimension $d=1$ was obtained by Makarov ([12],[13]).

\

\textbf{Theorem A} (Makarov, 1989)

\begin{enumerate}
\item If $f\in \Lambda_{*}(\R )$ then  $Dim ( \D (f) \cap I ) =1$ for
any interval $I\subset \R$. \item If  $f\in \lambda_{*}(\R )$ then
$Dim ( \D_0 (f) \cap I ) = 1$ for any interval $I\subset \R$.
\end{enumerate}

Here and hereafter, $Dim$ denotes  Hausdorff dimension.The authors
asked whether  $\D (f)$ (respectively  $\D_0 (f)$) should also have
maximal Hausdorff dimension if $f\in \Lambda_{*}(\R^d )$ (
respectively $f\in \lambda_{*}(\R^d )$). In a previous work ([7])
they showed that this is not the case: the right dimension is $1$
and this is the best that can be said in general.

\

\textbf{Theorem B} (D-LL-N, 2010)
\begin{enumerate}
\item If $f\in \Lambda_{*}(\R^d )$ then $Dim ( \D (f) \cap Q ) \geq
1$ for any cube $Q\subset \R^d$.
\item If $f\in \lambda_{*}(\R^d )$ then $Dim ( \D_0 (f) \cap Q ) \geq
1$ for any cube $Q\subset \R^d $.
\item There is $f\in \lambda_{*}(\R^d )$ such that $Dim( \D (f))=1$.
\end{enumerate}

In this paper we will mainly focus on differentiability properties
of Weierstrass type functions. Our method can be presented in two
steps: (i) give sufficient conditions on a function $f\in
\Lambda_{*}(\R^d )$ implying that $Dim (\D (f) ) = d$, (ii) show
that a certain class of Weierstrass type functions satisfies the
previous sufficient conditions.

Regarding i), our method is based on a principle that has been known
for a long time: there is  a correspondence between the
differentiability properties of a function $f: \R^d \to \R$ and the
boundary behavior of $\nabla F : \R^{d+1}_+ \to \R^{d+1}$, where $F:
\R^{d+1}_+ \to \R$ is the Poisson extension of $f$ to the upper
half-space $\R^{d+1}_+$. See the results in Section 2. Let $F:
\R^{d+1}_+ \to \R $ be a harmonic function. We say that the gradient
vector field $\nabla F: \R^{d+1}_+ \to \R^{d+1}$ is \textbf{Bloch}
(resp. \textbf{little Bloch}), denoted $\nabla F \in
\mathcal{B}(\R^{d+1}_+ )$ ( resp. $\nabla F \in
\mathcal{B}_{0}(\R^{d+1}_+ )$) if
$$
\sup_{(x,y)} y \, || HF(x,y)|| < \infty
$$
where $HF(x,y) = D(\nabla F (x,y) )$ is the Hessian of $F$ at
$(x,y)$ and the supremum is taken over all $(x,y)\in \R^{d+1}_+$
(resp. $\displaystyle  \sup_{x} y \, ||HF (x,y)|| \to 0$ as $y\to
0$). Geometrically, the Bloch condition says that the oscillation of
$\nabla F$ in regions of $\R^{d+1}_+$ of a fixed hyperbolic diameter
 is uniformly bounded (see proposition $2.1$).

% It turns out  that if $f\in \Lambda_{*}(\R^d )$ then  the
% differentiability properties of $f$ are closely related to the
% boundary behaviour of $\nabla F$( see Propositions $2.2$, $2.3$ and
% Corollary $2.4$ below ).
For $F: \R^{d+1}_+ \to \R$ harmonic,
define, analogously to $(1.1)$ and $(1.2)$:
\begin{align}
\D (\nabla F ) = & \{ x\in \R^d \, : \,  \limsup_{y\to 0} |\nabla
F (x,y)| < \infty \} \, , \\
\D_0 (\nabla F ) = & \{ x\in \R^d \, : \, \lim_{y\to 0} \nabla F
(x,y) \, \, \text{exists} \, \, \}
\end{align}

%From Proposition $2.3$ and Corollary $2.4$ it follows that the
Then Theorem A can be deduced from the following stronger result of
Makarov (see [13], Ch. II, section 5).

\

\textbf{Theorem C } Let $F: \R^2_+ \to \R $ be a harmonic function.
\begin{enumerate}
\item If $\nabla F \in \mathcal{B} (\R^2_+ )$,  then $Dim ( \D (\nabla F ) \cap I )
=1$ for any interval $I\subset \R$.
\item If $\nabla F \in \mathcal{B}_0 (\R^2_+ )$, then $Dim ( \D_0 (\nabla F ) \cap I )
=1$ for any interval $I\subset \R$.
\end{enumerate}

Corollary $2.4$ says that $\D (f) = \D ( \nabla_{x} F )$, where
$\nabla_x F $ denotes the tangential component of $\nabla F$.
Consequently Theorem C is stronger than Theorem A since the result
affects the two derivatives and  not only $ \frac{\partial
F}{\partial x}$. A decisive feature in the proof of Theorem C is
that, since $d=1$, then $\nabla F $ is an anti-analytic function so
in particular $HF(x,y)$ is a conformal matrix for each $(x,y)$. That
means that $\nabla F$ distorts in the same way over the different
directions, a fact which plays a rellevant role in the proof of
Theorem C.

More generally, we say that a smooth mapping $G: \Omega  \to \R^n$,
$\Omega \subset \R^n$ is \textbf{quasi-regular} if there exists a
constant $0<K<\infty$ such that
\begin{equation}
 \max_{|e| =1}|DG(x)(e)| \leq K \min_{|e| =1} |DG(x)(e)|
\end{equation}
for any $x\in \Omega$. Quasi-regularity implies that,
infinitesimally,  $G$ distorts about the same in the different
directions. See [2], [15] for an account on the theory of
quasi-regular mappings, under much milder regularity assumptions. If
$F$ is harmonic in $\R^2_+$ then $\nabla F$ is quasi-regular with
constant $1$. On the other hand, in dimensions greater or equal than
$3$, $K=1$ in $(1.5)$ implies that $DG(x)$ is a conformal matrix and
that $G$ is very rigid, according to a classical result of Liouville
(see [15]). In section $4$ we will use a weaker notion of
quasi-regularity. Let $F: \R^{d+1}_+ \to \R$ be a harmonic function.
 We say that $\nabla F$ is \textbf{weakly quasiregular} if there
exists an integer $N \geq 2$ and $\gamma \geq 1$ such that $\nabla
F$ satisfies the following $1/N$-\textbf{Weak Quasi-regularity}
condition with constant $\gamma $,
\begin{equation}
\int_{C_{1/N}(Q)} \max_{|e|=1} |(HF)e (x,y)|^2 dx dy \leq \gamma^2
\min_{|e|=1} \int_{C_{1/N}(Q)} |(HF)e (x,y) |^2 dx dy
\end{equation}
for any cube $Q\subset \R^d$ of sidelenght $l(Q)$, where $C_{1/N}(Q)
= Q \times [ l(Q) / N ,  l(Q)] \subset \R^{d+1}_+$ is the
$1/N$-Carleson box associated to the cube $Q \subset \R^{d}$,
$HF(x,y)$ is the Hessian of $F$ at $(x,y)$ and the maximum and the
minimum are taken over all unitary vectors $e \in \R^{d+1}$.
%$$
%||HF(x,y)||^2 = \sum_{i,j=1}^{d+1}\big | \frac{\partial^2
%F}{\partial x_i \partial x_j }(x,y) \big |^2
%$$
Obviously, quasi-regularity is stronger than weak quasi-regularity,
in the above sense. Our first main result says that, assuming weak
quasi-regularity, Theorem C can be generalized to any dimension.

\begin{thmn}
Let $F: \R^{d+1}_+ \to \R$ be a harmonic function. Assume that
$\nabla F$ is weakly quasiregular.
% satisfies a $1/N$-weak quasi-regularity condition with constant $
% \gamma \geq 1$ for some $N\in \N$, $N\geq 2$.
\begin{enumerate}
\item If $\nabla F \in \B (\R^{d+1}_+ ) $ then $Dim ( \D ( \nabla F ) \cap Q ) =
d$ for any cube $Q\subset \R^d$.
\item If $\nabla F \in \B_0 (\R^{d+1}_+ ) $ then $Dim ( \D_0 ( \nabla F ) \cap Q ) =
d$ for any cube $Q\subset \R^d$.
\end{enumerate}
\end{thmn}

As explained above, we obtain the following consequence which should
be compared with Theorem B.

\begin{corn} Let $f: \R^d \to \R$ be a bounded continuous function and let $F$ be its
Poisson extension to $\R^{d+1}_+$. Assume that $\nabla F$ is weakly
quasiregular.
% a $1/N$-weak quasi-regularity condition with constant $\gamma \geq
% 1$ for some $N\in \N$, $N\geq 2$.
\begin{enumerate}
\item If $f\in \Lambda_{*}(\R^d ) $ then $Dim ( \D (f) \cap Q ) =
d$ for any cube $Q\subset \R^d$.
\item If $f\in \lambda_{*}(\R^d ) $ then $Dim ( \D_0 (f) \cap Q ) =
d$ for any cube $Q\subset \R^d$.
\end{enumerate}
\end{corn}

The rest of the results deal with specific exemples of Zygmund
functions given by Weierstrass series. We have adapted ideas from a
recent paper of Y. Heurteaux ([10]) where he studies the nowhere
differentiability of Weierstrass-type functions in the real line.

% Let $b>1$ and  let $g\in \mathcal{C}^{1,1}(\R )$ be an  almost
% periodic function.  Define the Weierstrass function associated to
% $b$, $g$ as
% $$
% f_{b, g}(x) = \sum_{n=0}^{\infty} b^{-n} g(b^n x)
% $$
% In [9], Heurteaux proves the two following results:

% \textbf{Theorem E}.  Let $b$, $g$, $f_{b,g}$ be as above. Then there
% are two mutually exclusive cases:

% \begin{enumerate}
% \item  $f_{b,g} \in \mathcal{C}^{1,1}(\R ) $.
% \item  $f_{b, g}$ is nowhere differentiable.
% \end{enumerate}

% \ Actually Heurteaux proved something stronger that Theorem E:
% either $f_{b,g}$ is smooth or its second differences have large
% oscillation on any interval, which implies nowhere differentiability
% ( Theorem $1.2$ in [9]). He also gives the following sufficient
% condition.

% \

% \textbf{Theorem F}. Let $b$, $g$, $f_{b,g}$ be as above. Assume that
% either i) $g'(0) \neq 0$ or ii) $g$ is non constant and has a global
% extremum at $t=0$. Then $f_{b,g}$ is nowhere differentiable.

% \

% We have adapted Heurteaux arguments to the higher dimensional
% context. Let $\phi: \R^d \to \R$ such that $\phi \in \mathcal{C}^2
% (\R^d )$ and $\phi $ is $1$-periodic in each coordinate, that is:
% $$
% \phi (x_1 ,...,x_i +1,..., x_d ) = \phi (x_1 ,...,x_d )
% $$
% for any $x = (x_1 ,..., x_d )\in \R^d$ and each $i=1,...,d$.
For $\varepsilon >0$, let $\mathcal{C}^{1,\varepsilon}(\R^d )$ be
the class of bounded functions $f: \R^d \to \R$ for which there
exists a constant $C=C(f) >0 $ such that $|f(x+h) + f(x-h) - 2f(x)|
< C |h|^{1 + \varepsilon }$ for any $x, h \in \R^d$.  When $0<
\varepsilon <1$, the class $\mathcal{C}^{1,\varepsilon}(\R^d )$
consists of the differentiable functions whose first partial
derivatives belong to the H\"{o}lder class $Lip_{\varepsilon}
(\R^d)$. Also $\mathcal{C}^{2,\varepsilon}(\R^d )$ is the class of
functions whose first partial derivatives are in
$\mathcal{C}^{1,\varepsilon}(\R^d )$. Let $\phi: \R^d \to \R$ be a
function of class $\mathcal{C}^{1, \varepsilon} (\R^d )$ which is
 $1$-periodic in each coordinate, that is:
  $$
 \phi (x_1 ,...,x_i +1,..., x_d ) = \phi (x_1 ,...,x_d )
 $$
 for any $x = (x_1 ,..., x_d )\in \R^d$ and each $i=1,...,d$.
For $b>1$ define the Weierstrass function associated to $b$ and
$\phi$ by
\begin{equation}
f_{b, \phi}(x) = \sum_{n=0}^{\infty} b^{-n} \phi (b^n x)
\end{equation}
In  dimension $d=1$, Y. Heurteaux has proved in [10] that either
$f_{b, \phi} \in \mathcal{C}^{1,\varepsilon}(\R ) $ (and hence it is
differentiable at every point) or $f_{b, \phi}$ is nowhere
differentiable. This dichotonomy  result extends easily to dimension
$d>1$. Heurteaux also gives the following sufficient condition.

\textbf{Theorem D}. Let $d=1$,  $b>1$ and let $\phi$, $f_{b,\phi}$
be as above. Assume that either i) ${\phi}'(0) \neq 0$ or ii) $\phi$
is non constant and has a global
 extremum at $t=0$. Then $f_{b,\phi}$ is nowhere differentiable.

\

% It is easy to show that $f_{b,\phi} \in \Lambda_{*}(\R^d )$ and the
% proof goes exactly as in the case of the Weierstrass function:
% choose $h\in \R^d$ and suppose  that $b^{-(N+1)} < |h| \leq b^{-N}$
% for some integer $N$. Then, for each $x\in \R^d$,
% $$
% |f_{b, \phi}(x+h) + f_{b, \phi} (x-h) -2f_{b, \phi}(x)| \leq
% C||H\phi || \sum_{n=0}^N b^{-n} |h|^2  + ||\phi
% ||_{\infty}\sum_{n=N+1}^{\infty} b^{-n} \leq  C'|h|
% $$
% where
% $$
% ||HF(x,y)||^2 = \sum_{i,j=1}^{d}\big | \frac{\partial^2 F}{\partial
% x_i \partial x_j }(x,y) \big |^2
% $$
% The next result is the counterpart to Corollary $1.2$ in [9].
% \begin{thmn}
% Let $0< \alpha < 1$ and let $\phi \in \mathcal{C}^{2, \alpha}(\R^d )
% $ such that $\phi$ is $1$- periodic in each variable. If $b>1$, let
% $f_{b, \phi}$ be the Weierstrass function associated to $b$ and
% $\phi$ as in $(1.8)$. Then there are two mutually exclusive cases:
% \begin{enumerate}
% \item $f_{b, \phi} \in \mathcal{C}^{1, \epsilon}(\R^d )$ for any $0< \epsilon
% <1$.
% \item $f_{b, \phi}$ is nowhere differentiable.
% \end{enumerate}
% \end{thmn}

Similarly, we will say that a differentiable function $\phi: \R^d
\to \R $ satisfies condition $\mathcal H$ if, for each unitary
vector $e\in \R^d$, either $D_e \phi (0) \neq 0$ or the one-variable
function $t \to \phi (te )$ is nonconstant and has a global extremum
at $t=0$. Our main result is the following.

\begin{thmn}
Let $\phi: \R^d \to \R$ be  a function of class
$\mathcal{C}^{2,\alpha} (\R^d )$ for some $0< \alpha < 1$ which is
 $1$-periodic in each coordinate. For $b>1$ let $f_{b, \phi}$ be the
Weierstrass function associated to $b$ and $\phi$ as in $(1.7)$.
Assume in addition that $\phi$ satisfies condition $\mathcal H$.
Then
\begin{enumerate}
\item $f_{b, \phi}\in \Lambda_{*}(\R^d )$ and  $f_{b, \phi}$ is nowhere
differentiable.
% \item $m_d ( \D (f_{b,\phi}) ) = 0$.
\item For any unitary vector $e \in \R^d$ we have
$$
m_d \left\{x \in \R^d : \limsup_{t \to 0} \frac{|f_{b,\phi} (x +
te)- f_{b,\phi} (x)|}{|t|} < \infty \right\} =0
$$
In particular $m_d ( \D (f_{b, \phi}) =0$
\item $Dim ( \D (f_{b, \phi}) \cap Q ) = d$ for any cube $Q\subset \R^d$.
\end{enumerate}
\end{thmn}

\

The most relevant result in Theorem $3$ is part $3$ which should be
compared with Theorem B. The key point is to show that if $F$ is the
harmonic extension of $f_{b, \phi}$ to $\R^{d+1}_+$ then condition
$\mathcal H$ implies a certain uniform lower bound of  $HF$ (Lemma
$7.3$ below), which is the substitute of the oscillation condition
in Theorem $1.2$ in [10]. From such uniform lower bound, it is easy
to deduce that $\nabla F$ is weakly quasi-regular, which allows to
apply Theorem $1$.

\

The paper is organized as follows. Section $2$ contains some basic
facts about Zygmund functions and their connections with  harmonic
extensions. Section $3$ describes how to use stopping-time methods
to construct  Cantor-like  boundary sets at which a gradient Bloch
vector field is bounded. Section $4$ shows how  the weak
quasi-regularity condition guarantees that the boundary sets in
section $3$ have large Hausdorff dimension. In Section $5$, Theorem
$1$, part $1$ is proved and  an sketch of the proof of part $2$ is
given. Section $6$ contains some standard facts about regularity of
Poisson extensions. Section $7$ is devoted to functions of
Weierstrass type in higher dimensions. Theorem $3$ is proved in
section $8$. Finally, section $9$ includes some remarks and
questions.

\section{Some properties of Zygmund functions and their harmonic extensions }

We say that a harmonic function  $v$  in the upper half space
$\R^{d+1}_{+}$ is a \textit{Bloch } function, denoted $v\in \B
(\R^{d+1}_+ )$ , iff
$$
\sup \{ y |\nabla v (x,y)| : \, \, \,  (x,y)\in \R^{d+1}_+ \, \, \}
 = ||v||_{\B} < \infty
$$
If,
$$
\sup  y \{ |\nabla v (x,y)| : \, \, \,  x \in \R^d  \, \, \} \to 0
$$
as $y\to 0$ then we say that $v$ belongs to the \textit{Little Bloch
class}, denoted $v \in \B_0 (\R^{d+1}_+ ) $. The following
proposition is elementary
\begin{prop}
Let $v \in \B (\R^{d+1}_+ )$. If $a$, $b \in \R^d$ and  $s$, $t
> 0$ then we have
$$
| v(b,t) - v(a,s) | \leq  ||v||_{\B}\, \big ( \, \frac{|b-a|}{\max
\{ t, s \}} + | \log (\frac{t}{s})| \big )
$$
\end{prop}
\begin{proof}
Suppose that $0 < s \leq t$. Then $|v(b, t) - v(a,s) | \leq |v(b,t)
- v(a,t)| + |v(a,t) - v(a,s)|$. Use the Bloch condition on each
term.
\end{proof}

If $F$ is harmonic in $\R^{d+1}_+$ we say that $\nabla F \in \B
(\R^{d+1}_+ )$ (respectively $\nabla F \in \B_0 (\R^{d+1}_+ )$) if
all the partial derivatives $\frac{\partial F}{\partial x_i }$,
$\frac{\partial F}{\partial y}$, $i=1,...,d$ are Bloch (resp. Little
Bloch). Whenever $\nabla F \in \B (\R^{d+1}_+ )$, we also denote
$$
||\nabla F||_{\B} = \sup \{ y \big | \frac{\partial^2 F}{\partial
x_i
\partial x_j} (x,y) \big | \, : \, \, \, (x,y)\in \R^{d+1}_{+}\, , \, \, i,j= 1,...,d+1 \}
$$
where, for simplicity,  $x_{d+1}$ denotes the $y$-variable. The
following proposition whose proof can be found in [16, p. 146],
relates the Zygmund and Bloch classes

\begin{prop}
Let $f\in L^{\infty}(\R ^d )$ and let $F$ be its Poisson extension
to the upper half-space $\R^{d+1}_+ $. Then $f\in \Lambda_{*}(\R ^d
)$ if and only if $\nabla F \in \B (\R^{d+1}_+ )$. Moreover there
exists a positive constant $C$ only depending in $d$ such that
% $$
% f\in \Lambda_{*}(\R ^d ) \,  \, \, \Leftrightarrow \, \, \, \nabla F
% \in \B (\R^{d+1}_+ )
% $$
$$
C^{-1}||f||_{*} \leq ||\nabla F||_{\B} \leq C ||f||_{*}
$$
Furthermore, $ f\in \lambda_{*} (\R^d) $ if and only if $ \nabla F
\in \B_0 (\R^{d+1}_+ )$.
% $$
% f\in \lambda_{*} (\R^d) \, \, \Leftrightarrow \, \, \nabla F \in
% \B_0 (\R^{d+1}_+ )
% $$

\end{prop}

The following two propositions relate the incremental quotients of
Zygmund functions to the vertical  behaviour of the tangential
component of the gradient of their Poisson extensions. Given a
smooth function $F: \R^{d+1}_+ \rightarrow \R$, the tangential
component of its gradient is ${\nabla}_x F = (\frac{\partial
F}{\partial x_1}, \ldots , \frac{\partial F}{\partial x_d})$
\begin{prop}
Let $f\in \Lambda_{*}(\R^d )$ and let $F$ be its Poisson extension
to $\R^{d+1}_+$. Then, for any $x$, $h \in \R^d$, $h\neq 0$, we have
$$
 \frac{|f(x+h)- f(x) - h\cdot \nabla_x F (x, |h|)|}{|h|}\leq 4
||f||_*
$$
\end{prop}

\begin{proof}
We will use the following representation of $f$ , which can be
checked by differentiation: for any $y>0$,
\begin{equation}
f(x) = F(x,y) -y\frac{\partial F}{\partial y}(x,y) + \int_0^y
t\frac{\partial^2 F}{\partial y^2} (x,t) dt \label{minx}
\end{equation}
Choose $y=|h|$ to get that $f(x+h) -f(x) - h\cdot \nabla_x F(x, |h|)
= \Delta_1 - \Delta_2 + \Delta_3 $ where
\begin{align*}
\Delta_1  & = F(x+h, |h|) - F(x, |h|) - h\cdot \nabla_x F(x, |h|)\\
\Delta_2  & = |h| \big [ \frac{\partial F}{\partial y}(x+h, |h|)  -
\frac{\partial
F}{\partial y}(x, |h|) \big ] \\
\Delta_3 & =   \int_0^{|h|} t\big ( \frac{\partial^2 F}{\partial
y^2}(x+h, t) - \frac{\partial^2 F}{\partial y^2}(x, t)  \big ) dt
\end{align*}
Clearly $|\Delta_3| < 2 ||\nabla F||_{\B}|h| $ and Proposition
$(2.1)$ gives that $|\Delta_2| \leq ||\nabla F||_{\B}|h|$. On the
other hand, from Proposition $2.1$,
$$
|\Delta_1 | = \big | \int_0^1 h\cdot (\nabla_x F(x+th, |h|)-
\nabla_x F(x, |h|)dt  \big | \leq  ||\nabla F||_{\B}|h|
$$
\end{proof}

The following result follows easily.

\begin{cor}
Let $f\in \Lambda_{*}(\R^d) $. Then
$$
\limsup_{|h| \to 0}\frac{|f(x+h)- f(x)|}{|h|} < \infty \, \,
\Leftrightarrow \, \, \limsup_{y\to 0} |\nabla_x F( x, y)| < \infty
$$
\end{cor}

The analogues of Proposition $2.3$ and Corollary $2.4$ for the
little Zygmund class read as follows and are proved in the same way.
\begin{prop}
Let $f\in \lambda_{*}(\R^d )$ and let $F$ be its Poisson extension
to $\R^{d+1}_+$. Then for any $x \in \R^d$, one has
$$
\lim_{h\to 0} \frac{f(x+h) - f(x) - h\cdot \nabla_x F(x, |h|)}{|h|}
= 0
$$
\end{prop}
\begin{cor}
Let $f$, $F$ be as in  Proposition $2.5$. Then
$$
f \, \text{is differentiable at}\, \,  x  \, \, \Leftrightarrow \,
\, \lim_{y\to 0} \nabla_x F(x,y) \, \, \text{exists}
$$
\end{cor}
\section{On the boundedness of  Bloch gradients at the boundary. }

Let $F$ be harmonic in $\R^{d+1}_+$ such that $\nabla F$ is Bloch (
for instance $F$ could be the Poisson extension of a Zygmund
function in $\R^d$). We are interested in the size of the set
\begin{equation}
\D (\nabla F)= \{ x\in \R^d \, \, : \, \, \limsup_{y\to 0}|\nabla
F(x,y)| < \infty \} \label{minx}
\end{equation}
\begin{prop}
Let $F$ be a harmonic function in $\R^{d+1}_+$ such that $\nabla F
\in \B (\R^{d+1}_+ )$. Then
$$
\D (\nabla F) = \{ x\in \R^d \, : \lim_{y\to 0}\nabla F (x,y) \text{
exists} \} \cup N
$$
where $N$ has $d$-dimensional measure zero. In particular, if $F$ is
the Poisson extension of a function $f\in \Lambda_{*}(\R^d )$ and
$$
\limsup_{h\to 0}\frac{|f(x+h) - f(x)|}{|h|} = \infty \, \, \, \, \,
\,  a.e. \, \,  x\in \R^d \, ,
$$
then $\D (\nabla F)$ has zero $d$-dimensional measure.
\end{prop}
\begin{proof}
From Proposition $2.1$,  $\nabla F$ is non-tangentially bounded at
any point of $\D (\nabla F)$. From the Local Fatou Theorem for
harmonic functions ([16], Ch. VII, Theorem 3) it follows that for
almost all points  $x\in \D (\nabla F)$ , the limit $\lim_{y\to
0}\nabla F (x,y)$ exists. The second implication follows from
Corollary $2.4$.
\end{proof}
Since $\D (\nabla F)$ could have $d$-dimensional Lebesgue measure
zero, we may ask what can be said about its Hausdorff dimension. The
purpose of sections $3$ and $4$  is to establish that, if $\nabla F$
satisfies a certain quasiregularity assumption, then $\D (\nabla F
)$ has Hausdorff dimension $d$. In dimension $d=1$, this was proved
by N. Makarov ([12],[13]). Observe that if $d=1$, the
quasiregularity assumption is always fulfilled.

One can obtain satisfactory lower estimates of the Hausdorff
dimension of sets of Cantor type, as the following lemma shows. It
is a well known higher-dimensional version of a lemma of Hungerford
(see [14], Theorem 10.5). Hereafter,  $l(Q)$ stands for the
sidelength of a cube $Q$.
\begin{lem}
Let $\alpha , \beta > 0$ with $ \alpha < \beta^{1/d} < 1  $. For $k
= 0, 1,2...$, let $E_k$ be a countable union of disjoint cubes $\{
Q^k_j : j =1,2, \ldots \}$ in $\R^d$. Suppose that for any $k=0,1,2,
\ldots$, the following two conditions hold
\begin{enumerate}
\item
Whenever  $Q^{k+1}_i \cap Q^k_j \neq \emptyset$ then $Q^{k+1}_i
\subset Q^k_j$  and  \, $l(Q^{k+1}_i) < \alpha \, l(Q^k_j )$
\item
$\sum (l(Q^{k+1}_i ) )^d \geq \beta \, (l(Q^k_j ))^d $ where the sum
is taken over all cubes $Q^{k+1}_i$ such that $Q^{k+1}_i \subset
Q^k_j $.
\end{enumerate}
Then
$$
\dim \big ( \bigcap_{k=0}^\infty E_k \big )  \geq \frac{\log ( \beta
/\alpha^d ) }{\log ( 1/ \alpha ) }
$$
\end{lem}
In sections $3$ and $4$ we will show that, under certain
assumptions, $\D (\nabla F)$ contains Cantor-like subsets of
Hausdorff dimension arbitrarily close to $d$.

\

Pick an integer $N \geq 2$ that will remain fixed thorough the
section.  Let $Q\subset \R^d$ be a cube and $l= l(Q)$. Divide each
side of $Q$ into $N$ open-closed intervals of length $l/N$. In this
way we get $N^d$ disjoint cubes of sidelength $l/N$ whose union is
$Q$. We call them the \textit{N-adic descendents} of $Q$ of the
first generation. When repeating this to each descendent cube of the
first generation, we get $N^{2d}$ disjoint cubes of sidelength
$l/N^2$ whose union is $Q$, the $N$-adic descendents of $Q$ of
second generation. We denote by $\mathcal E_j (Q)$ the family of the
$N^{jd}$ descendent cubes of $Q$ of generation $j$. If $Q_j \in
\mathcal E_j (Q)$ then there is a chain $Q = Q_0 \supset Q_1 \supset
...\supset Q_j $ where $Q_i \in \mathcal D_i (Q)$, $i = 1,..., j$.
We call it the \textit{N-adic tower } from $Q_j$ to $Q$. Finally,
let $\mathcal E (Q) = \bigcup_{j=0}^\infty \mathcal E_j (Q)$ be the
family of all \textit{N-adic} subcubes of $Q$. The reason to use
$N$-adic divisions instead of just dyadic divisions is merely
technical.

Let $F$ be a harmonic function in $\R^{d+1}$ such that $\nabla F \in
\B (\R^{d+1}_+ )$. We describe now an stopping-time argument that
will produce a Cantor-like set contained in $\D (\nabla F )$. If
$Q\subset \R^d$ is a cube of sidelength $l(Q)$, we denote hereafter
\begin{align*}
\widehat{Q} = & \, \, Q \times [0, l(Q)] \\
(\nabla F )_Q =  &  \, \, \frac{1}{m_d (Q)} \int_{Q} \nabla  F( x,
l(Q))dx
\end{align*}
where the integral is understood in a vector-valued sense.

\

Fix a cube $Q\subset \R^d$. For any large positive number $M$ we
introduce a family $\mathcal S_M (Q)\subset \mathcal E (Q)$ of
$N$-adic subcubes of $Q$ in the following way. Given  $Q_j \in
\mathcal E_j (Q)$ whose $N$-adic tower is denoted by $Q = Q_0
\supset Q_1 \supset ...\supset Q_j $, we say that $Q_j \in \mathcal
S_M (Q)$ if and only if
$$
|(\nabla F )_Q - (\nabla F )_{Q_i} | \leq M \, \, , \, \, \,  i = 1,
..., j-1
$$
and
$$
|(\nabla F )_Q - (\nabla F )_{Q_j} | >M
$$
In other words, the family $\mathcal S_M (Q)$ consists of the
maximal $N$-adic subcubes $Q'$ of $Q$ which satisfy $|(\nabla F )_Q
- (\nabla F )_{Q_j} | >M$. The following proposition collects the
main facts about $\mathcal S_M (Q)$.
\begin{prop}
Let $F$  be harmonic in $\R^{d+1}_+$ such that $\nabla F \in \B
(\R^{d+1}_+ )$. Assume that $ m_d ( \D ( \nabla F )) = 0$. Then
there exists a  constant $C = C(N,d) >0$  such that
\begin{enumerate}
\item
For each $Q' \in \mathcal S_M (Q)$, we have
$$
M < |(\nabla F )_Q - ( \nabla F )_{Q'} | \leq M + C||\nabla F||_{\B}
$$
Furthermore, if $x\in Q'$, and $l(Q')\leq y \leq l(Q)$, then
$$
|\nabla F (x,y)-  (\nabla F )_Q | \leq  M + C||\nabla F ||_{\B}
$$
\item
$ \displaystyle m_d \Big ( Q \setminus   \bigcup_{ \mathcal S_M
(Q)}Q' \Big ) = 0 $.
\item
Each $Q'\in \mathcal S_M (Q)$ is a $N$-adic subcube of $Q$ of
generation at least $\frac{M}{C||\nabla F ||_{\B}}$
\end{enumerate}
\end{prop}
\begin{proof}
Part $1$ follows from proposition $2.1$.  Part $2$  is consequence
of $1$ and the assumption that $\D (\nabla F )$ has zero Lebesgue
measure. To show $3$, let $Q = Q_0 \supset Q_1 \supset ...\supset
Q_j = Q' $ be the $N$-adic tower from $Q'$ to $Q$, then
$$
M < |(\nabla F )_Q - (\nabla F )_{Q'}| \leq \sum_{i= 1}^j |(\nabla F
)_{Q_{i-1}} - (\nabla F)_{Q_i} | \leq C j ||\nabla F ||_{\B}
$$
\end{proof}

Given $a\in \R^{d+1} \setminus \{ 0 \}$ and $0< \theta < \pi /2$,
let
$$
\Gamma_{\theta}(a) = \{ b\in \R^{d+1} \, \, : \, \, (a-b)\cdot a >
|a|\, |a-b| \, \cos \theta  \}
$$
be the symmetric cone of vertex $a$ and aperture $2\theta$, whose
axis is the line containing the origin and the point $a$. Given $M >
0$ and  $0< \theta < \pi /2$, we introduce a subfamily of $\mathcal
S_M(Q)$, denoted by $\mathcal S_{M, \theta}(Q)$, that will play a
relevant role later. First, if $(\nabla F )_Q = 0$, we take
$\mathcal S_{M,\theta }(Q) = \mathcal S_M (Q)$. If $(\nabla F )_Q
\neq 0$, let $\xi_Q = - (\nabla F)_Q / |(\nabla F )_Q |$ and define
$\mathcal S_{M, \theta }(Q)$ to be the collection of all cubes
$Q'\in \mathcal S_M (Q)$ that satisfy
\begin{equation}
((\nabla F )_Q - (\nabla F )_{Q'}) \cdot \xi_Q
> \cos \theta \, | ((\nabla F )_Q - (\nabla F )_{Q'})| \label{minx}
\end{equation}
Observe that if $(\nabla F )_Q \neq 0$, then $\mathcal S_{M, \theta
}(Q)$ consists exactly of those cubes $Q' \in \mathcal S_M (Q)$ for
which $(\nabla F )_{Q'} \in \Gamma_{\theta}( (\nabla F )_Q )$.

We will need the following elementary lemmas.
\begin{lem}
Fix  $\frac{\pi}{3} \leq \theta < \frac{\pi}{2}$, $R\geq 0$, $k\geq
0$ such that $\displaystyle R\geq k / \cos \theta$. Let $a, b \in
\R^{d+1}$ such that $b \in \Gamma_{\theta}(a)$ if $a\neq 0$. Suppose
that  $ R \cos \theta \leq |a-b|  \leq R \cos \theta + k $

Then
\begin{equation}
|a| \leq R \Rightarrow |b| \leq \sqrt{R^2 \sin^2 \theta + k^2} \leq
R
\end{equation}
\end{lem}
\begin{proof}
From the cosine Theorem and the assumption $b \in
\Gamma_{\theta}(a)$ it is enough to compute the maximum of of
$g(x,y) = (x^2 + y^2 -2xy\cos \theta )^{1/2}$ in the rectangle $[0,
R ] \times [R\cos \theta , R\cos \theta + k]$. It is easy to check
that the maximum must be attained at one of the two corners $(0,
R\cos \theta + k)$, $(R, R\cos \theta + k )$ and that,
$\displaystyle g(0, R\cos \theta +k) \leq  g(R , R\cos \theta + k) =
\sqrt{R^2 \sin^2 \theta + k^2} $ provided $R \geq k / \cos \theta$,
$ \pi / 3 \leq \theta <  \pi / 2$. This proves the lemma.
\end{proof}

The following refinement of Lemma $3.4$ will be needed  to deal with
the little Bloch case.
\begin{lem}
Let $\frac{\pi}{3} \leq \theta < \frac{\pi}{2}$ and $g(x) =
\sqrt{x(x+1)}$, for $x\geq 0$.  Let $\{k_n \}$ be a bounded sequence
of positive numbers. Let $ R_1 \geq g(k_1 / \cos \theta) $ and let
$R_n$, $n = 1, 2...$ be a sequence of positive numbers defined
recursively by
\begin{equation}
R_{n+1} =  \max \big \{ \, g( \frac{k_{n+1}}{\cos \theta}) \, , \,
\sqrt{R_n^2 \sin^2 \theta + k_n^2} \, \big \}
\end{equation}
Then
$$
\limsup_{n} R_n = g \big( \frac{1}{\cos \theta}\, \limsup_n k_n \big
)
$$
In particular: if $k_n \to 0$ then $R_n \to 0$ but $R_n / k_n \to
\infty$.
\end{lem}
\begin{proof}
Let $k = \limsup_n \, k_n $, $R = \limsup_n R_n$. From the
construction of $R_n$ we have that $\{R_n \}$ is bounded and  $ R_n
\geq g(k_n / \cos \theta)$. Hence $R \geq g(k / \cos \theta ) $. On
the other hand
$$
R = \max \big \{ \, g( \frac{k}{\cos \theta} ) , \sqrt{R^2 \sin^2
\theta + k^2} \, \big \}
$$
which implies $R \leq g( k / \cos \theta )$, since $g(x) \geq x$ for
$x\geq 0$. If $k_n \to 0$, it follows that $R_n / k_n \to \infty$
since $R_n \geq g( k_n / \cos \theta )$.
\end{proof}

\begin{cor}
Let $R\geq 0$, $\pi / 3 \leq \theta < \pi /2$, $M = R \cos \theta$,
and $F$,$Q$,$C$,$ \mathcal S_M (Q)$ be as in Proposition $3.3$.
Consider the subamily $\mathcal S_{M, \theta}(Q)$ defined in (3.2).
 Assume that $\displaystyle R \geq C||\nabla F
||_{\B} / \cos \theta$. Then for each $Q' \in \mathcal S_{M, \theta}
(Q)$,
\begin{equation}
|(\nabla F )_Q | \leq R \, \, \Rightarrow \, \, |(\nabla F )_{Q'} |
\leq \sqrt{R^2 \sin^2 \theta + C^2 ||\nabla F ||^2_{\B} } \leq R
\label{minx}
\end{equation}
Furthermore, if $|(\nabla F )_Q | \leq R $ then
\begin{equation}
|\nabla F (x,y)| \leq (1 + \cos \theta )R + C ||\nabla F ||_{\B}
\end{equation}
whenever $x\in Q'$ and $l(Q') \leq y \leq l(Q)$.
\end{cor}
\begin{proof}
Take $a = (\nabla F )_Q$, $b = (\nabla F)_{Q'}$ and $k = C||\nabla F
||_{\B}$. By Proposition $3.3$, $R\cos \theta \leq |a -b|  \leq R
\cos \theta + k $. Since $R \geq k / \cos \theta$, $(3.5)$ follows
from Lemma $3.4$ . Inequality $(3.6)$ also follows from part $1$ of
Proposition $3.3$.
\end{proof}

Now we are ready to construct a Cantor-type set that will be
contained in $\D (\nabla F) $. Start with a fixed cube $Q_0 \subset
\R^d$. Fix $\pi /3 \leq \theta < \pi /2$ and $M>0$. For each $k\geq
0$ we will define a family of cubes $\mathcal G_k$ as follows: let
$\mathcal G_0 = \{ Q_0 \}$ and $\mathcal G_1 = \mathcal S_M (Q_0 )$.
If $k\geq 2$ and $\mathcal G_{k-1}$ has already been defined, we
take
$$
\mathcal G_k  = \bigcup_{Q\in \mathcal G_{k-1}}\, \mathcal S_{M,
\theta}(Q)
$$
Let
\begin{equation}
E_k = \bigcup_{Q\in \mathcal G_k }Q \, \, ,  \, \, \, \,  E_{\infty}
= \bigcap_{k = 0}^{\infty} E_k   \label{minx}
\end{equation}
Observe that from the construction and  Proposition $3.3$,
assumption $1$ of Lemma $3.2$ is satisfied with $\alpha = N^{-
\frac{M}{C||\nabla F ||_{\B}}}$.

\begin{prop} Fix $\pi / 3 \leq  \theta <
\pi / 2$ , $R>0$ and $M = R\cos \theta$.  Let $F$, $C$ be as in
Proposition $3.3$. Given  a cube $Q_0 \subset \R^d$,  construct
$E_{\infty}$ as in $(3.7)$. Suppose that
\begin{equation}
R \geq  \max \big \{ \frac{C||\nabla F ||_{\B} }{\cos \theta}\, , \,
\,  |(\nabla F )_{Q_0}|  \}
\end{equation}
Then
$$
E_{\infty} \subset \{ x\in Q_0 \, : \, \, \, \sup_{0 < y \leq l(Q_0
)}|\nabla F(x,y)| \leq 2R \}
$$
\end{prop}
\begin{proof}

Observe first that from $(3.5)$ in Corollary $3.6$, we deduce that
if $k \geq 0$ and  $Q\in \mathcal G_k$ then $|(\nabla F )_Q | \leq
R$. Therefore, from  Corollary $3.6$ and the assumption on $\theta$
we get
$$
| \nabla F (x,y) | \leq ( 1 + \cos \theta )R + C||\nabla F||_{\B}
\leq 2R
$$
if $x\in E_{\infty}$ and $0 < y \leq l(Q_0 )$.
\end{proof}
\begin{cor}
Let $F$ be a harmonic function in $\R^{d+1}_+$ such that $\nabla F
\in \B ( \R^{d+1}_+ )$. Fix $ \pi /3 \leq \theta < \pi /2$, $R>0$,
$M = R\cos \theta$ and let $C>0$ be the constant appearing in
Proposition 3.3. Suppose that there is a constant $ 0 < \beta < 1$
such that for any cube $Q_0 \subset \R^d$,  any subcube $Q\subset
Q_0$ and any $R$ satisfying
 $$
R \geq  \max \big \{ \frac{2C||\nabla F ||_{\B}}{\cos \theta} \, ,
\, \, |(\nabla F )_{Q_0}| \,  , \, \,
    \frac{C||\nabla F ||_{\B}\log (\frac{1}{\beta})}{d \log N} \big \}
$$
we have
\begin{equation}
\sum_{Q'\in \mathcal S_{M, \theta }(Q)} \, (l(Q'))^d \geq \beta
(l(Q))^d \label{minx}
\end{equation}
Then
$$
Dim \{x\in Q_0 \, : \, \, \, \sup_{0 < y \leq l(Q_0 )} | \nabla F
(x,y)| \leq 2R \, \} \geq d - \frac{C ||\nabla F ||_{\B}\log (
\frac{1}{\beta})}{R ( \cos \theta )  \log N}
$$
In particular,
$$
Dim ( \D ( \nabla F ) \cap Q_0 ) = d
$$
for any cube $Q_0 \subset \R^d$.
\end{cor}
\begin{proof}
If $\beta$ is as in $(3.9)$ and $\alpha = N^{-\frac{M}{C||\nabla
F||_{\B}}}$ the result follows from Proposition $3.7$ and Lemma
$3.2$.

\end{proof}

\section{A weak quasiregularity condition of harmonic gradients. }

In this section the notion of weak quasiregularity is discussed and
it is applied to get the relevant estimate, which is stated in
Corollary $4.5$, in the stopping time process described in Section
$3$. This estimate is needed in the proof of Theorem $1$.

We recall some well known facts from elementary linear algebra.
Suppose that $A = (a_{ij})$ is a $(d+1)\times (d+1)$ symmetric
matrix. Then
$$
\min |\lambda_i | = \min_{|e| =1}|Ae| \leq \max_{|e| =1}|Ae| = \max
|\lambda_i |
$$
where $\{ \lambda_i \}$ are the eigenvalues of $A$. Therefore, $A$
distorts exactly in the same way along all directions if and only if
$A$ is a conformal matrix, that is, all its eigenvalues have the
same absolute value. On the other hand, observe that
$$
||A ||^2 = \sum_{i,j =1}^{d+1} a_{ij}^2 = \sum_{i=1}^{d+1} |Ae_i |^2
$$
where $e_1$, ..., $e_{d+1}$ is the canonical basis of $\R^{d+1}$.
Furthermore,
$$
\frac{1}{\sqrt{d+1}}||A || \leq \max_{|e| =1}|Ae| \leq  ||A||
$$

If $F: \R^{d+1}_{+} \to \R$,  let $HF$ be the Hessian matrix of $F$,
that is, the $(d+1) \times (d+1)$ - matrix of all second derivatives
of $F$. If $\xi \in  \R^{d+1}$ we interpret $(HF) \xi$ as the
function obtained by matrix multiplication. Furthermore, according
to the previous comments we denote
\begin{equation}
||HF(x,y)||^2 = \sum_{i,j=1}^{d+1}\big | \frac{\partial^2
F}{\partial x_i \partial x_j }(x,y) \big |^2
\end{equation}
Here we denote $x = (x_1 , \ldots, x_d) \in \R^d$ and $x_{d+1} = y$.
 Fix $0 < \delta < 1$. For each cube $Q\subset \R^d$ we define the
$\delta$-\textit{hyperbolic box} associated to $Q$ as
$$
C_{\delta}(Q) = \{ (x, y) \in \R^{d+1}_+ \, : \, \, x\in Q \, , \,
\, \delta l(Q) \leq y \leq l(Q) \}
$$

%$N\in \N$. For each cube $Q\subset \R^d$  we define the
%\textit{N-hyperbolic box} associated to $Q$ as
%$$
%C_{1/N}(Q) = \{ (x, y) \, : \, \, x\in Q \, , \, \, \frac{l(Q)}{N}
%\leq y \leq l(Q) \}
%$$
Let $0 < \delta < 1 \leq \gamma$. We say that $\nabla F$ satisfies a
$\delta$- \textbf{weak Quasi-Regularity condition} with constant
$\gamma$ if  for any cube $Q\subset \R^d$ the following inequality
holds
\begin{equation}
 \int_{C_{\delta}(Q)} \max_{|e|=1} |(HF)e (x,y)|^2 dx dy \leq \gamma^2
\min_{|e|=1} \int_{C_{\delta}(Q)} |(HF)e (x,y) |^2 dx dy
\label{minx}
\end{equation}
where the $\max$ and $\min$ are taken over all unitary vectors $e\in
\R^{d+1}$. If we are not interested in the particular value of the
constant $\gamma$ we will just say that $\nabla F$ satisfies a
$\delta$- \textbf{weak QR condition}. For technical reasons we will
only use in this section values of $\delta $ of the form  $ 1/N$ for
an integer $N \geq 2$.

\

 Now, fix an integer $N\geq 2$, a sufficiently large number $M>0$ and a cube
 $Q\subset \R^d$.  Let $\mathcal S_M
(Q)$  be the family of $N$-adic subcubes of $Q$ introduced  in
section $3$. For $j\geq 1$, let $\mathcal S_M^j (Q)$ be the family
of all cubes $Q' \in \mathcal S_M (Q)$ of generation at most $j$.
Define also
\begin{align*}
T(Q) = & \{ (x,l(Q))\, : \, \, x\in Q \} \,  , \\
T_j (Q) = & \bigcup_{Q'\in \mathcal S_M^j (Q)} T(Q') \, , \\
S_j (Q) = & \bigcup_{Q'\in \mathcal S_M^j (Q)} Q' \, , \\
B_j (Q) = & \{ (x, N^{-j}l(Q)) \, : \, \, x\in Q \setminus S_j (Q) \} \, , \\
\Omega_j (Q) = & \, ( \widehat{Q} \setminus \bigcup_{Q'\in \mathcal
S_M^j (Q)} \widehat{Q'} ) \cap \{ (x,y)\, : \, \, x\in Q\, , \, \,
N^{-j}l(Q) < y < l(Q) \}
\end{align*}

Note that $\Omega_j (Q)$ is a sort of truncated domain associated to
the stopping time originating $\mathcal S_M (Q)$. In order to show
that the cubes in $\mathcal S_{M, \theta}(Q)$ take  a fixed amount
of the $d$-dimensional measure of $Q$ (as required in assumption
$(3.9)$) we will need to assume that $\nabla F $ satisfies a QR
condition. The technique basically consists of using Green's formula
applied to the functions $y$ and $|\nabla F - (\nabla F )_Q |^2 $ in
the domains $\Omega_j$. We will prove  a sequence of technical
lemmas. In the rest of the section, $m_d (E)$ stands for the
$d$-dimensional Lebesgue measure of $E\subset \R^d$.

\begin{lem}
Let $F$ be harmonic in $\R^{d+1}_{+}$ such that $\nabla F \in
\mathcal{B}(\R^{d+1}_+)$ and $m_d ( \D (\nabla F )) = 0$. Let $M>
0$, $Q$ and $\Omega_j (Q)$ be as above. Then there is a  constant $C
= C(d,N)>0$ such that if $M
> C||\nabla F ||_{\mathcal{B}}$, then
\begin{equation}
\int_{\Omega_j (Q)} y || HF (x,y)||^2 dxdy \geq \frac{1}{8} M^2 \,
(l(Q))^d \label{minx}
\end{equation}
for sufficiently large $j$.
\end{lem}
\begin{proof}
Note that $\Omega_j = \Omega_j (Q)$ is bounded by a finite number of
smooth hypersurfaces. We apply Green's formula to $y$ and $|(\nabla
F - (\nabla F )_Q |^2$ in $\Omega_j$. By simple computation,
\begin{align*}
\nabla (|(\nabla F - (\nabla F )_Q |^2 ) = & \, 2 (HF)(\nabla F -
(\nabla F)_Q ) \, , \\
\triangle (|(\nabla F - (\nabla F )_Q |^2 ) = & \, 2 ||HF ||^2
\end{align*}
So , from Green's formula, the integral in $(4.3)$ is equal to
\begin{equation}
\int_{\partial \Omega_j }y (HF)(\nabla F - (\nabla F )_Q ) \cdot n -
\frac{1}{2}\int_{\partial \Omega_j }|\nabla F - (\nabla F )_Q |^2
\nabla (y)\cdot n \label{minx}
\end{equation}
where $n$ denotes the outer normal unit vector. Observe first that
by construction of $\Omega_j$ and Proposition $3.3$, we have
\begin{align}
|\nabla F - (\nabla F )_Q | & \leq M + C ||\nabla F ||_{\B}  \, \,
\, \, \, \text{on}  \, \, \, \, \, & \Omega_j  \\
|\nabla F - (\nabla F )_Q | & \geq M - C ||\nabla F ||_{\B}  \, \,
\, \, \,  \text{ on}  \, \, \, \, \,  & \bigcup_{Q'\in \mathcal
S_M^j (Q)}T(Q')
\end{align}
with $C = C(d, N)$.  From the Bloch condition, $(4.5)$ and the fact
that the surface measure of $\partial \Omega_j$ is smaller than a
fixed multiple of $(l(Q))^d$, it follows
\begin{equation}
\big | \int_{\partial \Omega_j } y (HF)(\nabla F - (\nabla F )_Q
)\cdot n \big | \leq C ||\nabla F ||_{\B}( M + C||\nabla F ||_{\B})
\, (l(Q))^d \label{minx}
\end{equation}
where $C = C(d, N)$. On the other hand,
$$
\int_{\partial \Omega_j} | \nabla F - (\nabla F )_Q |^2 \nabla
(y)\cdot n = \int_{T(Q)}| \nabla F - (\nabla F )_Q  |^2 - \int_{T_j
(Q) \cup B_j (Q)} | \nabla F - (\nabla F )_Q |^2
$$
From the Bloch condition we get
\begin{align}
\int_{T(Q)} | \nabla F - (\nabla F )_Q |^2 & \leq ( C||\nabla F
||_{\B})^2 (l(Q))^d \\
\int_{B_j (Q)}| \nabla F - (\nabla F )_Q |^2  & \leq (M + C||\nabla
F ||_{\B})^2 ( (l(Q))^d - m_d (S_j (Q) ) \, .
\end{align}
On the other hand, from $(4.6)$ we get
\begin{equation}
\int_{T_j (Q)}| \nabla F - (\nabla F )_Q |^2 \geq ( M - C||\nabla F
||_{\B})^2 m_d ( S_j (Q)) \label{minx}
\end{equation}
Now, by part $2$ of Proposition $3.3$ we have that   $m_d ( S_j (Q))
\to (l(Q))^d$ as $j\to \infty$. Choose $j$ large enough so that $m_d
(S_j (Q)) > \frac{3}{4}(l(Q))^d$. It then follows, from
$(4.7)$-$(4.9)$ and simple computation that there is a positive
constant
 $C= C(d,N)$ such that
\begin{align*}
\big| \int_{\partial \Omega_j} |\nabla F - (\nabla F )_Q |^2 \nabla
(y)\cdot n \big | & \geq ( \frac{M^2}{2} - 3C||\nabla F ||_{\B}M
-\frac{3}{2} C^2 ||\nabla F ||_{\B}^2 ) (l(Q))^d \\
& \geq \frac{M^2}{4} (l(Q))^d
\end{align*}
where the last inequality holds as soon as $ M > 14C||\nabla F
||_{\B}$. This proves the lemma.
\end{proof}

We need now a variant of Lemma $4.1$.
\begin{lem}
Let $F$,$Q$,$M$ and $\Omega_j (Q)$ be as in Lemma $4.1$. Assume that
$\nabla F$ satisfies a weak  $1/N$-QR condition with constant
$\gamma \geq 1$ for some integer $N \geq 2$. Then there is a
constant $C=C(d,N)>0$ such that if $M
> 36\gamma^2 C||\nabla u||_{\B}$  then
$$
\min_{|e| =1} \int_{T_j (Q)} \big [ (\nabla F - (\nabla F )_Q )\cdot
e \big]^2 \geq \frac{M^2}{16\gamma^2}(l(Q))^d
$$
for sufficiently large $j$, where the minimum is taken over all
unitary vectors $e\in \R^{d+1}$.
\end{lem}
\begin{proof}
The proof mimics Lemma $4.1$. Fix a unitary vector $e \in \R^{d+1}$.
Consider the harmonic function $v = (\nabla F - (\nabla F )_Q )
\cdot e$. We apply Green's formula in $\Omega_j = \Omega_j (Q)$ to
the functions $y$ and $v^2$. A simple computation shows
\begin{align*}
\nabla v^2  = \, \, & 2((\nabla F - (\nabla F )_Q )\cdot e ) \,
(HF)e
\\
\triangle (v^2) = \, \, & 2 |(HF)e |^2
\end{align*}
Therefore, as in lemma $4.1$,
\begin{align*}
& \int_{\Omega_j }2 y |(HF)e |^2  = \\ & =   \int_{\partial \Omega_j
}2 y (\nabla F - (\nabla F )_Q \cdot e) \, (HF)e \cdot n  -
\int_{\partial \Omega_j}\big [ (\nabla F - (\nabla F )_Q )\cdot e
\big]^2 \nabla (y)\cdot n
\end{align*}
Now observe that, since $\Omega_j $ can be decomposed in a union of
disjoint hyperbolic boxes of the form $C_{1/N}(Q')$ we have, from
the $1/N$-QR condition and Lemma $4.1$:
\begin{equation}
\frac{M^2}{8} (l(Q))^d \leq \int_{\Omega_j} y ||HF||^2 \leq 2
\gamma^2 \int_{\Omega_j } y |(HF)e |^2
\end{equation}
On the other hand, as in Lemma $4.1$,
\begin{equation}
  | \int_{\partial \Omega_j } y (\nabla F - (\nabla F )_Q )\cdot
e  (HF)e \cdot n  |   \leq C ||\nabla F ||_{\B}( M + C||\nabla F
||_{\B}) \, (l(Q))^d \label{minx}
\end{equation}
As for the other surface integral, notice that $\nabla (y)\cdot n$
vanishes outside the horizontal part of $\partial \Omega_j $, which
consists of $T_j (Q) \cup T(Q) \cup B_j (Q)$. Furthermore,
\begin{align}
& \int_{T(Q)} \big [(\nabla F - (\nabla F )_Q )\cdot e \big ]^2 \leq
C^2 ||\nabla F ||_{\B}^2 (l(Q))^d \\
& \int_{B_j (Q)}\big [(\nabla F - (\nabla F )_Q )\cdot e \big ]^2
\leq (M + C||\nabla F ||_{\B})^2 m_d (B_j (Q))
\end{align}
If $j$ is large enough , $ m_d (B_j (Q))$ can be made arbitarily
small. Then, combining $(4.11)$- $(4.14)$,  we finally get
$$
\int_{T_j (Q)}\big [(\nabla F - (\nabla F )_Q )\cdot e \big ]^2 \geq
\frac{M^2}{16\gamma^2}(l(Q))^d
$$
as soon as $M> 36\gamma^2 C||\nabla F ||_{\B}$.
\end{proof}
Let $0< \theta  < \pi /2$ and $e \in \R^{d+1}$ be a unitary vector.
Define
\begin{align*}
& \mathcal S_{M, e , \theta}^{j , +}(Q) = \{ Q' \in \mathcal S_{M}^j
(Q) \, : \, \,  ((\nabla F )_{Q'} - (\nabla F )_Q )\cdot e
> \, \, \, \, \, \cos \theta | (\nabla F )_{Q'} - (\nabla F )_Q | \} \\
& \mathcal S_{M, e , \theta}^{j , -}(Q) = \{ Q' \in \mathcal S_{M}^j
(Q) \, : \, \,  ((\nabla F )_{Q'} - (\nabla F )_Q )\cdot e
< - \cos \theta | (\nabla F )_{Q'} - (\nabla F )_Q | \} \\
& \mathcal S_{M, e , \theta}^{j }(Q) = \, \mathcal S_{M, e ,
\theta}^{j , +}(Q) \cup \mathcal S_{M, e , \theta}^{j , -}(Q)
\end{align*}
If $(\nabla F )_Q \neq 0$ and we make the particular choice $e  = -
\frac{(\nabla F )_Q}{|(\nabla F )_Q |}$ we recover the family
$\mathcal S_{M, \theta}$ introduced in in section $3$:
\begin{equation}
\mathcal S_{M, \theta} (Q) = \bigcup_{j}\mathcal S_{M, e ,
\theta}^{j , +}(Q)
\end{equation}
\begin{lem}
Assume  that $\nabla F $ satisfies a weak $1/N$-QR condition with
constant $\gamma \geq 1$. There is a constant $C =C(d,N)>0$ such
that if $M
> 36 \gamma^2 C ||\nabla F ||_{\B}$ and $\cos \theta \leq 1 / 8
\gamma$  then
$$
\sum_{Q' \in \mathcal S_{M, e , \theta}^{j}(Q)} ( l(Q'))^d \geq
\frac{1}{150 \gamma^2} (l(Q))^d
$$
for any unitary vector $e\in \R^{d+1}$.
\end{lem}
\begin{proof}
If $Q' \in \mathcal S_M^j (Q) \setminus \mathcal S_{M, e , \theta}^j
(Q) $ then on $T(Q')$ we have the estimate
$$
|(\nabla F - (\nabla F )_Q )\cdot e | \leq C||\nabla F ||_{\B} +
\cos \theta ( M + C||\nabla F ||_{\B} )
$$
while, if $Q' \in \mathcal S_{M, e , \theta}^j (Q)$ then on $T(Q')$:
$$|(\nabla F - (\nabla F )_Q )\cdot e | \leq M +
C||\nabla F ||_{\B}
$$
Therefore
\begin{align*}
 \int_{T_j (Q)}\big [(\nabla F - (\nabla F )_Q )\cdot e \big ]^2 \leq
& \big [ C||\nabla F ||_{\B} + \cos \theta ( M + C ||\nabla F||_{\B}
) \big ]^2 (l(Q))^d + \\
+ &( M + C ||\nabla F ||_{\B} )^2 \sum_{Q' \in \mathcal S_{M, e ,
\theta}^j (Q)} (l(Q'))^d
\end{align*}
From the last inequality and Lemma $4.2$ we obtain
$$
\sum_{Q' \in \mathcal S_{M, e , \theta}^j (Q)} ( l(Q'))^d \geq
\frac{\frac{1}{16\gamma^2} - [ \cos \theta + (1 + \cos \theta )
\frac{C||\nabla F ||_{\B} }{M}]^2} {(1 + \frac{C ||\nabla F
||_{\B}}{M})^2} ( l(Q))^d \geq \frac{1}{150\gamma^2} ( l(Q))^d
$$
where the last inequality holds provided $M > 32\gamma C||\nabla F
||_{\B}$ and $\cos \theta \leq \frac{1}{8\gamma}$.
\end{proof}

\begin{lem}
Let $N$, $F$, $Q$, $M$ and $\Omega_j $ as in Lemma $4.1$. Suppose
that $\nabla F$ satisfies a weak $1/N$-QR condition with constant
$\gamma \geq 1$. Let $ 0 < \theta < \pi /2 $ such that $ \cos \theta
\leq 1 / 6400\gamma^3$. There is a constant $C=C(d,N)>0$ such that
for any unitary vector $e \in \R^{d+1}$ and sufficiently large $j$
we have
$$
\sum_{Q' \in \mathcal S_{M,e , \theta }^{j, +} (Q)} (l(Q'))^d \geq
\frac{(l(Q))^d}{16000\gamma^3}
$$
provided $M > 10^5 \gamma^3 C ||\nabla F ||_{\B}$.
\end{lem}
\begin{proof}
Fix $M$ and $e$. Choose $0< \theta_1 < \theta_2 = \theta $ such that
$\cos \theta_1 = \frac{1}{8\gamma}$. For $i=1,2,$ let $\mathcal
A_i^+ = \mathcal S_{M, e , \theta_{i}}^{j, +} (Q)$, $\mathcal A_i^-
= \mathcal S_{M, e , \theta_{i}}^{j,-} (Q)$ and $\mathcal A_i  =
\mathcal A_i^+ \cup \mathcal A_i^-$. Furthermore, denote
\begin{align*}
& A_i^+ = \bigcup_{Q' \in \mathcal A_i^+}T(Q')\\
& A_i^- = \bigcup_{Q' \in \mathcal A_i^-}T(Q')
\end{align*}
and $A_i = A_i^+ \cup A_i^-$, $H_i = T_j (Q) \setminus A_i$.  From
Lemma $4.3$, we have $m_d (A_1 ) = m_d (A_1^+ ) + m_d (A_1^- ) \geq
\frac{1}{150 \gamma^2}(l(Q))^d$. If $m_d (A_1^+ ) \geq
\frac{1}{300\gamma^2}(l(Q))^d$ we are done  so assume that $m_d
(A_1^- ) \geq \frac{1}{300\gamma^2}(l(Q))^d$. Since $|( (\nabla
F)_{Q'} - (\nabla F )_Q )\cdot e | > M \cos {\theta}_1$ for any $Q'
\in A_1^-$, we deduce that
\begin{equation}
\int_{A_1^- } |\nabla F - (\nabla F )_Q )\cdot e | \geq (\cos
\theta_1 M - C||\nabla F||_{\B} ) m_d (A_1^- ) \geq
\frac{M(l(Q))^d}{3200\gamma^3}
\end{equation}
where, for the last inequality we have used $\cos \theta_1 =
\frac{1}{8\gamma}$ and $M > 32\gamma C||\nabla F ||_{\B}$.

Now we apply Green's formula in $\Omega_j$ to the functions $y$ and
$v = (\nabla F - (\nabla F )_Q )\cdot e$. Note that $v$ is harmonic,
while $\nabla v = (HF)e$. Therefore, from Green's formula:
\begin{equation}
\int_{\partial \Omega_j}( \nabla F - (\nabla F)_Q )\cdot e \nabla
(y)\cdot n = \int_{\partial \Omega_j}y (HF)e \cdot n
\end{equation}
As  in the previous lemmas, the integrand  on the left hand side of
$(4.17)$ vanishes outside $T_j (Q) \cup B_j (Q) \cup T(Q)$.
Furthermore
\begin{align}
& \big | \int_{\partial \Omega_j }y (HF)e \cdot n \big | \leq C
||\nabla F ||_{\B}(l(Q))^d \\
& \big | \int_{B_j (Q)} (\nabla F - (\nabla F)_Q )\cdot e \big |
\leq M m_d (B_j (Q) ) \\
& \big | \int_{T(Q)} (\nabla F - (\nabla F)_Q )\cdot e \big | \leq
C||\nabla F ||_{\B} (l(Q))^d
\end{align}
Now we use the decomposition $T_j (Q) = A_2^+ \cup A_2^- \cup H_2$
and observe that , since $A_1^- \subset A_2^-$, we get from $(4.16)$
\begin{equation}
\big | \int_{A_2^-} (\nabla F - (\nabla F)_Q )\cdot e \big | \geq
\frac{M}{3200\gamma^3}(l(Q))^d
\end{equation}
On the other hand
\begin{equation}
\big | \int_{H_2} (\nabla F - (\nabla F)_Q)\cdot e \big | \leq (\cos
\theta_2 ) ( M + C||\nabla F||_{\B} ) (l(Q))^d
\end{equation}
Now since $m_d (B_j (Q)) \to 0$ as $j\to \infty$ we get from
$(4.18)$-$(4.23)$ that if $j$ is large enough:
$$
\int_{A_2^+} (\nabla F - (\nabla F)_Q )\cdot e \geq  \big ( M (
\frac{1}{3200\gamma^3} - \cos \theta_2 ) - (2 + \cos \theta_2 )
C||\nabla F ||_{\B}\big ) (l(Q))^d
$$
From the last inequality and the fact that $|\nabla F - (\nabla F)_Q
| \leq M + C||\nabla F||_{\B}$ on $A_2^+$ we finally deduce
$$
m_d (A_2^+ ) \geq \frac{(l(Q))^d }{16000\gamma^3}
$$
provided $M > 10^5 \gamma^3 C||\nabla F||_{\B}$.
\end{proof}

Finally, we collect the previous estimates in the following
statement.

\begin{cor}
Let $F$ be harmonic in $\R^{d+1}_{+}$ such that $\nabla F \in
\mathcal{B}(\R^{d+1}_+)$ and $m_d ( \D (\nabla F )) = 0$. Assume
that $\nabla F $ satisfies a weak $1/N$- QR condition with constant
$\gamma \geq 1$ for some integer $N\geq 2$. Then there are constants
$0 < \theta_0 = \theta_0 (\gamma ) < \pi /2$, $0 < \beta = \beta
(\gamma ) < 1 $ and $C = C ( \gamma , N , d ) > 0 $ such that for
any $\theta $, $\theta_0 \leq \theta < \pi/2$, any cube $Q_0 \subset
\R^d$, any subcube $Q \subset Q_0$ and any $\displaystyle M \geq
C||\nabla F ||_{\B} $ we have
$$
\sum_{Q'\in \mathcal S_{M, \theta }(Q)} \, (l(Q'))^d \geq \beta
(l(Q))^d
$$
\end{cor}

\section{Proof of Theorem $1$  }

%\noindent{\textbf{Proof of Theorem $1$, part $1$:}}

\begin{proof} ( Part $1$)

Fix $\pi /3 \leq \theta  < \pi/2$ as in Corollary $4.5$,  a cube
$Q_0 \subset \R^d $ and assume that $m_d ( \D ( \nabla F ) \cap Q_0
) =0 $, otherwise the result is trivial. If $R$ is large enough and
$M = R\cos \theta$ then from Corollary $4.5$
$$
\sum_{Q' \in S_{M, \theta} (Q)} ( l(Q'))^d \geq \beta (l(Q))^d
$$
where $\beta = \beta (\gamma ) > 0$. Now if follows from Lemma $3.2$
and Corollary $3.8$ that
$$
Dim \{x\in Q_0 \, : \, \, \, \limsup _{y\to 0} | \nabla F (x,y)|
\leq 2R \} \geq d - \frac{C ||\nabla F ||_{\B}\log (
\frac{1}{\beta})}{R (\cos \theta ) \log N}
$$
and the result follows letting $R \to \infty$.

\end{proof}

Now we will adapt the  results in previous section to cover the case
of gradients in the little Bloch class (part $2$ in Theorem $1$).
First, we obtain an analogue of Plessner theorem on the boundary
values of analytic functions in the unit disc for gradients of
harmonic functions in the upper half-space which are weakly
quasiregular. If $f$ is analytic in the unit disc $\mathbb{D}$, a
classical result of Plessner ([14],Theorem 6.13) says that for
almost all points $e^{i\theta} \in
\partial \mathbb{D}$, either $f$ has a finite non-tangential limit at
$e^{i\theta}$ or the image by $f$ of any symmetric cone with vertex
$e^{i\theta}$ is dense in the complex plane $\mathbb{C}$. Therefore
the boundary behaviour of $f$ at a.e. $e^{i\theta}$ is either very
good or very bad. If $z = (x', y)\in \R^{d+1}_+ $ and $x\in \R^d$,
the notation $z \sphericalangle x$ means that $z$ tends to $x$
non-tangentially, that is, $z$ tends to $x$ and
$$
z \in \Gamma_{\alpha }(x) = \{( x',y) \: \, \, \, |x-x'| \leq y \tan
\alpha   \, \, , \, \, \, 0< y <1 \}
$$
for a fixed $0 < \alpha < \pi /2$. We refer to [16] for the main
results concerning non-tangential boundary behaviour of harmonic
functions in the upper half-space. The following proposition says
that if a harmonic gradient is weakly quasiregular, then Plessner
theorem still holds. Observe that, as in the analytic case, no
growth assumption is required.

\begin{prop}
Let $u$ be harmonic in $\R^{d+1}_+$ such that $\nabla u$ is weakly
quaisregular. Then, for any cube $Q\subset \R^d$ one of the two
possibilities holds:
\begin{enumerate}
\item
$$
m_d \big ( \{ x\in Q \, :  \lim_{z\sphericalangle x} \nabla u (z)
\text{ exists}  \}  \big ) > 0
$$
\item
For any $0 < \alpha < \pi /2$ and for almost every $x \in Q$,
$\nabla u ( \Gamma_{\alpha}(x))$ is dense in $\R^{d+1}$.  In
particular, for any $a\in \R^{d+1}$ and for a.e. $x\in Q$,
$$
\liminf_{z\sphericalangle x} | \nabla u (z) - a | = 0
$$
\end{enumerate}
\end{prop}
\begin{proof}
Fix a cube $Q\subset \R^d$. Assume that part 2 does not hold.
Standard measure theoretic arguments show that there are a set $E
\subset Q$ with $m_d (E)
> 0$, $a\in \R^{d+1}$,  $0 < {\alpha}_0 < \pi / 2$ and $0 < y_0 < 1$ such that
$$
\inf |\nabla u (z) - a | > 0
$$
Here the infimum is taken over the set $\{z \in \R^{d+1}_+ : z \in
\Gamma_{\alpha_0}(x) , x \in E \} \cap \{0 <y < y_0 \}$. For
simplicity,  $E$ may denote hereafter different subsets of $Q$ of
positive $d$-dimensional measure. Choose $i\in \{1,..., d+1 \} $
such that
$$
\inf \big |\frac{\partial u}{\partial x_i}(z) - a_i \big |
> 0
$$
where, as usual, $x_{d+1} = y$ and $a = (a_1 , ..., a_{d+1})$. L.
Carleson proved that a harmonic function in $\R^{d+1}_+$ which is
non-tangentially bounded from below on a certain set of points in
$\R^d$, has a non-tangential limit at almost every point of this set
(see [5]).
 We deduce that $\lim_{z\sphericalangle x} \frac{\partial u}{\partial x_i}(z)$
exists for almost every $x\in E$.
%   Now, from the local Fatou theorem for  harmonic functions
% in the upper half-space ([15], Ch. VII, Theorem 3) we deduce that
% $\lim_{z\sphericalangle x} \frac{\partial u}{\partial x_i}(z)$
% exists for a.e. $x\in E$.
From well known results relating the boundary behaviour and the area
function of harmonic functions in the upper half-space ([16], Ch.
VII, Theorem 4), we deduce that
\begin{equation}
A^{2}_{\alpha}(\frac{\partial u}{\partial x_i })(x) =
\int_{\Gamma_{\alpha}(x)} y^{1-d}\big |\nabla (\frac{\partial
u}{\partial x_i}(z)) \big |^2 dz < \infty
\end{equation}
for a.e $x\in E$ and any $0 < \alpha  < \pi /2$. Here $A_{\alpha}v
(x)$ denotes the so called area function of $v$ associated to the
aperture $\alpha$. ( See [16] for details). Let $x$ and $\alpha$ be
as in $(5.1)$. Assume that $\nabla u$ satisfies a $\delta$ weak
quasiregular condition with constant $ \gamma$ for a certain $0<
\delta < 1$. Consider the truncated cones
$$
\Gamma_{\alpha , n}(x) = \Gamma_{\alpha}(x) \cap \{ (x' , y) \, : \,
\, \,  \delta^{(n+1) } \leq y \leq \delta^{n}  \}
$$
Then, if $Q'$ is the cube centered at $x$ with $l(Q') = 2\delta^{n}$
we have
$$
\Gamma_{\beta , n} (x)  \subset C_{\delta}(Q') \subset
\Gamma_{\alpha , n} (x)
$$
where $0< \beta < \alpha $ and $\tan \beta = \delta \tan \alpha$.
Therefore , since $\nabla (\frac{\partial u}{\partial x_i}) = (Hu)
e_i $ where $e_i$ is the $i$-th vector of the canonical basis in
$\R^{d+1}$, we get from condition $\delta$- \textit{QR}
$$
\int_{\Gamma_{\beta , n}(x)}y^{1-d}||Hu||^2  \leq C_1 \gamma^2
\int_{C_{\delta}(Q')} y^{1-d}\big | \nabla (\frac{\partial
u}{\partial x_i} ) \big |^2 \leq  C_1 \gamma^2 \int_{\Gamma_{\alpha
, n} (x)} y^{1-d} \big | \nabla (\frac{\partial u}{\partial x_i})
\big |^2
$$
where $C_1$ only depends on $\delta$, $\alpha$ and $d$. Therefore,
$$
A_{\beta}^{2}(\frac{\partial u}{\partial x_j}) (x) < \infty
$$
for all $j= 1, ..., d+1$. Using again the results relating the
non-tangential limits and the area function, this time in the
opposite direction, we finally get that $\displaystyle
\lim_{z\sphericalangle x} \nabla u (z) $ exists for almost every $x
\in E$. This proves the Proposition.
\end{proof}

\noindent{\textbf{Sketch of proof of Theorem $1$, part $2$:}}

\

Suppose that $F: \R^{d+1}_+ \to \R$ is harmonic and $\nabla F \in
\B_0 (\R^{d+1}_+ )$ satisfies a $1/N$- QR condition with constant
$\gamma \geq 1$. Pick a cube $Q_0 \subset \R^d$ and assume that $m_d
( \D_0 ( \nabla F ) \cap Q_0 ) = 0$ since otherwise the conclusion
is obvious. We will actually show that,
\begin{equation}
Dim \{  x\in Q_0 \, : \, \lim_{y\to 0} \nabla F (x,y) = a \} = d
\end{equation}
for each $a\in \R^{d+1}$. So in particular $Dim \D_0 ( \nabla F ) =
d$.

The proof is similar to the proof of Theorem $1$, a), the main
modifications coming from the following two key facts:
\begin{itemize}
\item In part 1 of Theorem $1$, we could assume that
$m_d ( \D (\nabla F ) \cap Q_0 ) = 0 $ so when we ran the stopping
time argument,  $\nabla F $ must escape balls with full $d$-measure
( see part $2$ in proposition $3.3$). In part $2$ of Theorem $1$,
the hypothesis is $m_d ( \D_0 (\nabla F ) \cap Q_0 ) = 0 $. By
Proposition $3.1$, $m_d ( \D (\nabla F ) \cap Q_0 ) = 0 $. Therefore
we can run the same type of stopping-time arguments using even small
balls.
% By the local Fatou Theorem for harmonic functions ([16], Ch.
% VII, Theorem 3), we deduce that  $m_d ( \D (\nabla F ) \cap Q_0 ) =
% 0 $. Therefore we can run the same type of stopping-time arguments
% using even small balls.

% However, in the case of Theorem $1$ b), the hypothesis $m_d ( \D_0
% (\nabla F ) \cap Q_0 ) = 0 $ does not imply a priori that $\nabla F$
% must escape balls with full $d$- measure.  But if $\nabla F$
% satisfies a weak QR condition, Proposition $5.1$ shows that this is
% actually the case. Therefore we can run the same type of
% stopping-time arguments using even small balls.

\item Related to the previous point is the fact that the meaning of $||\nabla F||_{\B} $ should now be relaxed,  in
the sense that it must be understood as a variable quantity that
tends to $0$ as long as we approach the boundary of $\R^{d+1}_+$.
This implies that  $k = C||\nabla F ||_{\B}$ in sections $3$ and $4$
can  be replaced now by a sequence $k_n \to 0$,  where $n$ refers to
the successive stopping-time steps in the construction.
Consequently,  it follows from Lemma $3.5$ that the sequence of
radius $R_n $ can be chosen in such a way that $R_n \to 0$, but $R_n
/ k_n \to \infty$ which allows replacing boundedness of $\nabla F$
by convergence to a given point.
\end{itemize}
A brief sketch of the proof would run as follows:

It is enough to take $a = 0$ in $(5.2)$. Fix $\pi /3 \leq \theta <
\pi /2$ as in Corollary $4.5$.  Then we run a sequence of stopping
times corresponding to sequences $k_n$, $M_n = R_n \cos \theta$ that
can be defined as follows. Suppose that $k_j$, $R_j$ have already
been chosen,  for $j=1,..., n$. According to Proposition $3.3$, part
$3$, we define
\begin{align}
k_{n+1} = & \sup \{ y ||HF(x,y )|| \, : \, 0 < y \leq N^{- (\cos
\theta ) S_n} \, l(Q_0 ) \} \\
R_{n+1}  = &  \max \big \{ g \big ( \frac{k_{n+1}}{\cos \theta} \big
) , \sqrt{R_n^2 \sin^2 \theta + k_n^2 } \big \}
\end{align}
where $S_n = \sum_{j=1}^n R_j / k_j$. Then $k_n \to 0$ and, from
Lemma $3.5$, $R_n \to 0$ but $ R_n / k_n \to \infty$. Then the same
arguments as in section $3$ can be used to obtain that
$$
Dim \{ x\in Q_0 \, : \, \lim_{y\to 0} \nabla F (x,y ) = 0 \} = d
$$

\section{Some estimates for Poisson integrals.}
In this section we collect some estimates relating the regularity of
a function defined on $\R^d$ to the regularity of its Poisson
extension to the upper half-space $\R^{d+1}_+$.

For $0< \alpha < 1$ let $\Lambda_\alpha (\R^d)$ be the H\"{o}lder
class of bounded functions $f : \R^d \to \R$ for which
$$
\|f \|_\alpha = \sup_{x, h \in \R^d} \frac{|f(x+h) -
f(x)|}{\|h\|^\alpha} < \infty \, .
$$
The class $C^{1,  \alpha} (\R^d)$ ( resp. $C^{2, \alpha} (\R^d)$ )
consists of those differentiable functions defined in $\R^d$ whose
first partial derivatives (resp. second partial derivatives) belong
to the H\"{o}lder class $\Lambda_{\alpha} (\R^d)$.

The following two propositions will be needed later.

\begin{prop}
Let $\phi : \R^d \to \R$ such that $\phi \in C^{1,  \alpha} (\R^d)$
and  let $\Phi$ be the Poisson extension of $\phi $ to $\R^{d+1}_+$.
Then there exists a positive constant $C$, depending only on $d$ and
$\alpha$ such that
$$
|\nabla \Phi (x,y)| \leq C ( ||\phi ||_{\infty} + ||\nabla \phi
||_{\infty} + ||\nabla \phi ||_{\alpha} )
$$
for all $(x,y)\in \R^{d+1}_+$.
\end{prop}
\begin{proof}
First, observe that $\frac{\partial \Phi}{\partial x_i}$ is the
Poisson extension of $\frac{\partial \phi}{\partial x_i}$, $i= 1,
\ldots , d$. This shows that
$$
||\frac{\partial \Phi}{\partial x_i}||_{\infty} \leq
||\frac{\partial \phi}{\partial x_i}||_{\infty}
$$
if $i = 1,..., d$. Now we estimate $\frac{\partial \Phi}{\partial
y}$. For the case $y\geq 1$, use, as in section $2$,  the
representation
$$
\frac{\partial \Phi}{\partial y}(x,y) = \frac{1}{2}\int_{\R^d}P_y
(z,y) [\phi ( x+z) + \phi (x-z) - 2\phi (x)]dz
$$
Then
$$
|\frac{\partial \Phi}{\partial y}(x,y)| \leq A_d ||\phi
||_{\infty}\int_{\R^d}\frac{dz}{( 1 + |z|^2 )^{\frac{d+1}{2}}} \leq
B_d ||\phi ||_{\infty}
$$

Now assume $y<1$. Since for any $i=1, 2, \ldots , d$, the function
$\frac{\partial \phi}{\partial x_i}$ is in  $\Lambda_\alpha (\R^d)$,
its Poisson extension $\frac{\partial \Phi}{\partial x_i}$ satisfies
$$
\sup \{ y^{1- \alpha} \|\nabla (\frac{\partial \Phi}{\partial x_i}
(x,y)) \|: (x,y) \in \R_{+}^{d+1} \} < C(d) (||\nabla \phi
||_{\alpha} + ||\nabla \phi ||_{\infty}) \, .
$$
(See [16, p.142]). Hence
$$
\sup \{ y^{1- \alpha} \| \frac{{\partial}^2 \Phi}{{\partial}^2 y}
(x,y) \|: (x,y) \in \R_{+}^{d+1} \} < C_1 (d) (||\nabla \phi
||_{\alpha} + ||\nabla \phi ||_{\infty}) \, .
$$
Integrating we deduce that if $0 < y <1 $, then $\frac{\partial
\Phi}{\partial y}$ is uniformly bounded by $C_2 (\alpha , d)
(||\nabla \phi ||_{\alpha} + ||\nabla \phi ||_{\infty})$.   Actually
the proof gives the stronger result
$$
\sup \{ y^{1- \alpha} \| {\nabla}^2  \Phi (x,y) \|: (x,y) \in
\R_{+}^{d+1} \} < \infty .
$$

\end{proof}

% If $P(x,y)$, $x\in \R^d , \, y>0$ is the Poisson kernel in
% $\R^{d+1}_+$ and $P_y (x,y)$ denotes differentiation on $y$ then
% $$
% P_y (x, y) = C_d \frac{|x|^2 -dy^2}{(|x|^2 + y^2 )^{\frac{d+3}{2}}}
% $$
% and, as in section $1$,
% \begin{equation}
% \frac{\partial \Phi}{\partial y}(x,y) = \frac{1}{2}\int_{\R^d}P_y
% (z,y)[\phi (x+z) + \phi (x-z) - 2\phi (x)] dz
% \end{equation}
% By the Lipschitz assumption on $\frac{\partial \phi}{\partial x_i}$
% we have
% \begin{equation}
% |\phi (x +z) + \phi (x-z) -2\phi (x)| \leq ||\nabla \phi
% ||_{Lip_1}\, |z|^2
% \end{equation}
% Suppose that  $0<y<1$. We split the integral in $(5.1)$ into three
% parts $A$ , $B$, $C$ corresponding, respectively, to integration on
% $0< |z| < y$ , $y \leq |z| \leq 1$ and $|z| > 1$. Then from $(5.2)$
% and the hypothesis,
% \begin{align*}
% & |A| \leq C_d ||\nabla \phi ||_{Lip_1}\int_{0}^{y}\frac{y^2
% r^{d+1}}{y^{d+3}} dr \leq C_d ||\nabla \phi ||_{Lip_1}
% \\
% & |B| \leq C_d ||\nabla \phi
% ||_{Lip_1}\int_{y}{1}\frac{r^{d+3}}{r^{d+3}}dr \leq C_d ||\nabla
% \phi
% ||_{Lip_1} \\
% & |C| \leq C_d ||\phi ||_{\infty}
% \int_{1}^{\infty}\frac{r^{d+1}}{r^{d+3}} dr \leq C_d ||\phi
% ||_{\infty}
% \end{align*}
% The case $y\geq $ follows in a similar way.

Let $\phi : \R^d \to \R$ be a bounded function, let $\Phi$ be the
Poisson extension of $\phi$ to $\R^{d+1}_+$ and $H\Phi (x,y)$ the
$(d+1)\times (d+1)$- Hessian matrix of $\Phi$ at $(x,y)$. Let $||
H\Phi (x,y) ||$ be as in $(4.1)$. The next proposition provides
sufficient conditions on $\phi$ for the boundedness of $||H\Phi ||$.
Let $||H \phi||_\alpha$ be the sum of the Hölder norms of all second
order partial derivatives of the function $\phi$
\begin{prop}
Let $\phi $ be a function in the class $C^{2, \alpha} (\R^d)$. Then
there exists a positive constant $C$ depending only on $d$ and
$\alpha$, such that
$$
 ||H\Phi (x,y) || \leq C ( ||\nabla \phi
 ||_{\infty} + ||H\phi || + ||H\phi ||_{\alpha} )
$$
for any $(x,y)\in \R^{d+1}_+$.
\end{prop}
\begin{proof}
It is easy to check that $\frac{\partial \Phi}{\partial x_i}$ (resp.
$\frac{\partial^2 \Phi}{\partial x_i \partial x_j }$) is the Poisson
integral of $\frac{\partial \phi}{\partial x_i}$ (resp.
$\frac{\partial^2 \phi}{\partial x_i \partial x_j}$), $i=1, \ldots ,
d$. It follows then that $ ||\frac{\partial^2 \Phi}{\partial x_i
\partial x_j} ||_{\infty} = ||\frac{\partial^2 \phi}{\partial x_i
x_j}||_{\infty} < \infty $, $i, j = 1 , \ldots , d$. For
$\frac{\partial^2 \Phi}{\partial y
\partial x_i}$, apply Proposition $6.1$ to $\frac{\partial
\phi}{\partial x_i}$ instead of $\phi$. Finally, the estimate for
$\frac{\partial^2 \Phi}{\partial y^2}$ follows from harmonicity.
\end{proof}

We will need the following result which says that the Poisson
extension of a periodic function and its derivatives decay
exponentially  at infinity. It is valid under more general
assumptions but this version will be sufficient for applications.
\begin{prop}
Let $\phi: \R^d \to \R $ be a function in the class $C^{2,\alpha}
(\R^d)$. Assume that $\phi$ is $1$-periodic in each variable, that
is: $\phi ( x + e_i ) = \phi (x)$ for all $x\in \R^d$ and any vector
$e_i$ of the canonical basis in $\R^d$, $i = 1,...,d$. Let $\Phi
(x,y)$ be the Poisson extension of $\phi$ to $\R^{d+1}_+$. Then
there are positive constants $C_1 = C_1 (d, \alpha )$, $C_2 = C_2 (d
, \alpha )$ such that, for any $(x,y) \in \R^{d+1}_+$,
\begin{align}
& |\Phi (x,y) - C \int_{Q_0}\phi (z) dz | \leq C_1 ||\phi ||_{\infty}e^{-C_2 y} \\
& |\nabla \Phi (x,y) | \leq C_1 ( ||\phi ||_{\infty} + ||\nabla \phi ||_{\infty} + || \nabla \phi  ||_\alpha  )e^{-C_2 y} \\
& || H\Phi (x,y) || \leq C_1 ( ||\nabla \phi ||_{\infty} + || H\phi
||_\infty  + ||H\phi ||_{\alpha})  e^{-C_2 y}
\end{align}
where $Q_0 = [0,1]^d$ is the unit cube in $\R^d$.
\end{prop}
\begin{proof}

We use Poisson's summation formula (see [8]):
$$
\sum_{k\in \Z^d} \frac{y}{(|x + k|^2 + y^2 )^{\frac{d+1}{2}}} =
\sum_{k\in \Z^d} e^{-2\pi y |k| } \, e^{-2\pi i <k,x>}
$$
that holds for any $(x, y)\in \R^{d+1}_+$. Then, by periodicity,
$$
\Phi (x,y) = C_d \int_{Q_0} \sum_{k\in Z^d} \frac{y}{(|x-z+k|^2 +
y^2 )^{\frac{d+1}{2}}}  \phi (z)dz
$$
so, from  Poisson's formula:
\begin{equation}
\Phi (x,y)  = C_d \big ( \int_{Q_0}\phi (z)dz + \sum_{k\in \Z^d
\setminus \{ 0 \} } e^{-2\pi y|k|} \int_{Q_0} e^{-2\pi i <k, x-z>}
\phi (z) dz \, \big )
\end{equation}
Suppose first that $y\geq 1$. Then
$$
1 + \sum_{k\in \Z^d  }e^{-2\pi y|k|}  \leq \sum_{k\in \Z^d }e^{-2\pi
y( |k_1 | +...+|k_d | )/d} = \big ( \sum_{n= -\infty}^{\infty}
e^{-2\pi y|n|/d} \big )^d = \Big ( \frac{1 + e^{-2\pi y/d}}{1 -
e^{-2\pi y/d}}\Big )^d
$$
and therefore,
$$
|\Phi (x,y) - C_d  \int_{Q_0}\phi (z)dz | \leq ||\phi
||_{\infty}\Big [ \Big ( \frac{1 + e^{-2\pi y/d}}{1 - e^{-2\pi
y/d}}\Big )^d - 1 \Big ] \leq C_1 e^{-C_2 y}||\phi ||_{\infty}
$$
as soon as $y \geq 1 $, where $C_1$, $C_2$ depend only on $d$. If
$0< y<1$, the Maximum Principle gives $|\Phi (x,y)| \leq ||\phi
||_{\infty}$. This proves $(6.1)$. To prove $(6.2)$, we
differentiate in $(6.4)$:
$$
|\nabla \Phi (x,y) | \leq ||\phi ||_{\infty} \sum_{k \in \Z^d
\setminus \{0 \} } |k| e^{-2\pi y |k|}
$$
Suppose that $y\geq 1$. Then , since $te^{-2\pi yt} \leq e^{-\pi
yt}$ if $t\geq 1$ we have
$$
1 + \sum_{k \in \Z^d \setminus \{0 \}}|k|e^{-2\pi y |k|} \leq
\sum_{k\in \Z^d} e^{-\pi y |k|} \leq C_1 e^{-C_2 y}
$$
as above. If $0< y<1$, use Proposition $6.1$. Essentially the same
argument, together with Proposition $6.2$, proves $(6.3)$.
\end{proof}

Lemma $6.5$ below relates the  differentiability properties of a
function $f$ in $\R^d$ to the boundary behaviour of the Hessian of
its harmonic extension F to the upper-half space $\R^{d+1}_+$.
Suppose that $F: \R^{d+1}_+ \to \R$ is smooth, $t>0$ and $e\in
\R^{d+1}\setminus \R^d$ is  unitary. Then it is easy to check that

\begin{align}
 \frac{\partial}{\partial t} ( F(x+te))  & = \nabla F (x +te )\cdot
e \\
\frac{\partial ^2}{\partial t^2} ( F(x + te ))  & = (e^T (HF)e) (x
+te)
\end{align}
where $HF$ is the Hessian matrix of $F$ and we interpret $e^T (HF) e
$ in matrix form , $e^T$ being the vector $e$ written in row form.
We need the following technical proposition, which is a
generalization of formula $(2.1)$.

\begin{prop}
Let $f\in \Lambda_{*} (\R^d )$ and $F$ its Poisson extension to
$\R^{d+1}_{+}$.  Let  $e\in \R^{d+1}\setminus \R^d$ be an unitary
vector. Then for any $x\in \R^d$ and any $y> 0$ the following
identity holds
\begin{equation}
f(x) = \int_0^y t (e^T (HF)e)(x+te)dt - y \nabla F (x +ye)\cdot e +
F(x + ye)
\end{equation}
\end{prop}
\begin{proof}
Note that, from  $(6.5)$ and $(6.6)$, the right hand side term in
$(6.7)$ is independent of $y$. On the other hand, the Zygmund
condition and Proposition $2.2$ give
$$
| (e^T (HF)e)(x +te) | \leq \frac{C||f||_* }{t}
$$
and
$$
|\nabla F (x +ye) \cdot e | \leq C ||f||_* \log (\frac{1}{y} )
$$
where $C$ is a positive constant depending on $d$ and $e$. The lemma
follows letting $y$ tend to $0$.
\end{proof}

\begin{lem}
Let $f\in \Lambda_* (\R^d ) $ and let  $F$ be the Poisson extension
of $f$ to $\R^{d+1}_+$. Suppose that there are a unitary vector $e
\in \R^{d+1}$ and an open cube  $Q \subset \R^d$ for which
\begin{equation}
\sup_{(x, y) \in Q \times (0, 1]} |(HF)e (x,y)| < \infty \, .
\end{equation}
Then
\begin{enumerate}
\item
If $e\in \R^d$, both limits
$$
\lim_{t\to 0}\frac{f(x + te)-f(x)}{t} \, \, , \, \, \, \lim_{t\to 0}
\nabla F (x, t) \cdot e
$$
exist for any $x\in Q$.
\item
If $e\in \R^{d+1}\setminus \R^d$ then for any $i =1,...,d$
$$
\frac{\partial f}{\partial x_i}(x) \, \, , \, \, \, \, \lim_{t\to 0}
\frac{\partial F}{\partial x_i}( x +te)
$$
exist for any $x\in Q$.
\end{enumerate}
\end{lem}
\begin{proof}
From $(6.8)$, for any $x \in Q$, one has
$$
\sup_{0< y<1}| \frac{\partial}{\partial y}\nabla F(x,y)\cdot e | <
\infty
$$
which implies the existence of $\lim_{y\to 0}\nabla F (x,y)\cdot e$.

Assume first that $e\in \R^d$. Apply Proposition $6.4$ to the normal
unit vector in the $y$-direction. Then
$$
f(x +te) -f(x) = I_1 + I_2 + I_3
$$
where
\begin{align*}
& I_1 = \int_{0}^t y \big [ \frac{\partial ^2 F}{\partial
y^2}(x+te, y)- \frac{\partial^2 F}{\partial y^2}(x, y) \big ] dy \, , \\
& I_2 = t \big [  \frac{\partial F}{\partial y}(x+te, t) -
\frac{\partial F}{\partial y}(x, t) \big ] \, , \\
& I_3 = F(x+te, t)-F(x,t)
\end{align*}
From $(6.8)$ and the Cauchy estimates for harmonic functions it
follows that
$$\sup_{Q\times (0,1]} y | \frac{\partial^2}{\partial y^2}\nabla
F (x,y)\cdot e | < \infty \, , $$
 which implies that $|I_1 | \leq
Ct^2$. By $(6.8)$, $ |I_2 | \leq C t^2$ so both $I_1$ and $I_2$ are
$o(t)$ as $t\to 0$.  Finally, observe that
$$
\frac{I_3}{t} = \frac{1}{t}\int_0^t \nabla F ( x + se, t) \cdot e ds
$$
and that, again from $(6.8)$,
$$
| \nabla F ( x + se, t) - \nabla F ( x, t) | \leq Cs
$$
Therefore we deduce that $I_3 / t $ has a limit as $ t\to 0$ and
part i)  follows.

Now assume $e\in \R^{d+1} \setminus \R^d$. Fix $i=1, \ldots , d$.
From $(6.8)$ we get
$$
|\frac{\partial F}{\partial x_i}(x + te) - \frac{\partial
F}{\partial x_i} ( x + se) | \leq C |t-s|
$$
so $\lim_{t\to 0}\frac{\partial F}{\partial x_i}( x +te)$ exists.
Now for $t>0$, choose $M= M(t) > 0$  so that $M(t)\to \infty$ and
$tM(t)\to 0 $ as $t\to 0$. From Proposition $6.4$ we get
\begin{align*}
 f(x +te_i ) - f(x)  & = \int_{0}^{Mt}s \big [ (e^T (HF)e)(x +te_i +se ) - (e^T (HF)e)(x+se) ] ds \\
& - tM \big [  \nabla F (x + te_i +tMe)\cdot e - \nabla F (x +
tMe)\cdot  e  \big ] \\ & + F(x+te_i + tMe )-F(x + tMe)  = I_1 - I_2
+I_3
\end{align*}
As previously, Cauchy's estimates give that, $| (e^T (HF)e)(x +te_i
+se) - (e^T (HF)e)(x+se)| \leq  Ct/s$ so $|I_1 | \leq CMt^2$. In the
same way, $|I_2 | \leq CM t^2$. As for $I_3$,
$$
\frac{I_3}{t} = \frac{1}{t}\int_0^t \frac{\partial F}{\partial x_i }
(x +se_i +tMe )ds
$$
and, from Propositions $2.1$ and $2.2$,
$$
|\frac{\partial F}{\partial x_i}(x +se_i + tMe) - \frac{\partial
F}{\partial x_i}(x + tMe)| \leq C||f||_* \frac{s}{tM} \leq
\frac{C||f||_*}{M}
$$
Therefore for any $0<s<t$,
$$\frac{I_3}{t} = \frac{\partial F}{\partial x_i}( x + tMe)
+ O(\frac{1}{M})
$$
and the lemma follows from the choice of $M$ and the existence of
$\lim_{t \to 0}\frac{\partial F}{\partial x_i}( x + te)$.
\end{proof}

\section{Functions of Weierstrass type.}

Now we turn to functions of Weierstrass type. Fix $b>1$ and let
$\phi: \R^d \to \R $ be as in the statement of Theorem $3$. Let $f =
f_{b, \phi}$ be the Weierstrass function associated to $b$, $\phi$:
\begin{equation}
f(x) = \sum_{n=0}^{\infty} b^{-n} \, \phi (b^n x)
\end{equation}
Denote by $\Phi$ (resp. $F$) the Poisson extension of $\phi$ (resp.
$f$) to $\R^{d+1}_+$. Then
\begin{equation}
F(x,y) = \sum_{n=0}^{\infty} b^{-n} \Phi (b^n x , b^n y)
\end{equation}
which leads to the following functional equations, that will be used
later:
\begin{equation}
F(b x , by ) = b \, F (x, y) - b \Phi ( x, y)
\end{equation}
and by differentiation:
\begin{align}
\nabla F (b x, b y ) = &  \, \nabla F ( x,y ) - \nabla \Phi
( x, y) \\
b \, HF (b x , b y) = & \, HF(x,y ) - H\Phi (x,y)
\end{align}

\begin{prop}
Let $f$ (resp. $F$) be  as in $(7.1)$ (resp. $(7.2)$).  Then $f\in
\Lambda_* (\R^d )$ and
$$
 ||f||_*  \leq C_1 ( ||\phi ||_\infty + ||H\phi ||_\infty )
$$
Furthermore,
$$
 \sup \{ y \, ||HF(x,y)|| \, \, : \, \, \, (x,y)\in \R^{d+1}_+ \}  \leq
C_2 ( ||\phi ||_{\infty} + ||H\phi ||_\infty )
$$
where  $C_1$ and $C_2$ are positive constants  depending only on $d$
and $b$.
\end{prop}
\begin{proof}
The fact that  $f \in \Lambda_{*}(\R^d )$ uses an standard trick in
the theory of lacunary series. If $h \in \R^d $, denote $\Delta_h^2
f(x) = f(x+h) + f(x-h) -2f(x)$. If $|h| \geq 1 $, we have
$$
||\Delta_h^2 f ||_{\infty} \leq  4 ||f||_{\infty} \leq C(b) |h|
||\phi ||_{\infty}
$$
so we can assume that $|h| < 1$. Choose $N\geq 0$ so that $b^{-(N
+1)} < |h| \leq b^{-N}$. A simple computation shows that if $a$,
$h\in \R^d$ then
\begin{equation}
|\Delta_h^2 \phi (a) | \leq C(d)\, ||H \phi||_\infty \, |h|^2
\end{equation}
Now split $\Delta_h^2 f(x) = A + B $ where
$$
A = \sum_{n=0}^N b^{-n}\Delta_{b^n h}^2 \phi (b^n x ) \, \, , \, \,
\, B =\sum_{n=N +1}^{\infty} b^{-n}\Delta_{b^n h}^2 \phi (b^n x )
$$
From $(7.6)$ and the choice of $N$ we get that $|A| \leq C(d, b )|h|
||H\phi ||_\infty$. On the other hand, $|B| \leq 4
b^{-N}(b-1)^{-1}||\phi ||_{\infty} \leq C(b) |h| ||\phi ||_{\infty}
$. Hence $||f||_*  \leq C_1 ( ||\phi ||_\infty + ||H\phi ||_\infty
)$. The last assertion follows from  Proposition $2.2$.
\end{proof}

We remind now condition $\mathcal{H}$ on $\phi$, which, in the case
$d=1$, was used by Heurteaux ( see Theorem $3.1$ in [10]). We say
that $\phi $ satisfies condition $\mathcal H$ if, for any unitary
vector $e\in \R^d$, either $D_e \phi (0) \neq 0$ or the one-variable
function $t\to \phi (te)$,  is non constant and has a global
extremum at $t=0$.

\

\begin{prop} Let $b > 1$ and let $\phi$ be as in Proposition $6.3$.
Let $f = f_{b, \phi}$ be the Weierstrass-type function associated to
$b$, $\phi$ defined by $(7.1)$ and let $F$ be the Poisson extension
of $f $ to $\R^{d+1}_+ $. Assume in addition that $\phi$ satisfies
condition $\mathcal H$. Then for any $M> 0$  there exists $\eta  $,
$0< \eta < 1$ such that
$$
\inf_{|e| = 1} \, \sup_{(x,y) \in Q_0 \times [\eta , 1]}|(HF)e(x,y)
|
> M
$$
where $Q_0 = [-1/2, 1/2 ]^d$ and the infimum is taken over all
unitary $e\in \R^{d+1}$.
\end{prop}
\begin{proof}
Suppose that the conclusion does not hold. Then there is $M>0$ and a
sequence $\{ e_n \}$ of unitary vectors in $\R^{d+1}$ such that
$$
\sup_{(x,y) \in Q_0 \times [1/n, 1]}|(HF)e_n (x,y) | \leq  M
$$
Taking a subsequence if necessary we can assume that $e_n \to e$ ,
where $e\in \R^{d+1}$ is also unitary and satisfies
$$
\sup_{(x,y) \in Q_0 \times (0, 1]}|(HF)e(x,y)| \leq  M
$$
We will apply Lemma $6.5$ only in the case $x =0$. If $e\in \R^d$
then, from Lemma $6.5$, functional equation $(7.4)$ and the fact
that $\nabla \Phi $ is continuous up to the boundary we get
$$
\nabla \Phi (0,0) \cdot e = \nabla \phi (0)\cdot e = 0
$$
The rest of the argument follows [10]. By condition $\mathcal H$,
the function $\phi (te )$ is non-constant and has a global extremum
at $t = 0$, say a global maximum. In particular, $f(te)$ has also a
global maximum at $t=0$ and, from Lemma $6.5$,
$$
\lim_{t\to 0} \frac{f(te) - f(0)}{t} = 0
$$

Fix $t \in \R$. Now for each positive integer $n$,
$$
f(b^{-n}te ) -f(0) = \sum_{k=0}^{\infty} b^{-k}\big ( \phi
(b^{k-n}te) - \phi (0) \big ) \leq b^{-n} \big ( \phi (te) - \phi
(0) \big )
$$
Therefore
$$
0 = \lim_{n\to \infty }\frac{f(b^{-n}te ) - f(0) }{b^{-n}} \leq \phi
(te) - \phi (0)
$$
which contradicts the fact that $\phi (te)$ is non-constant and has
a global maximum at $t=0$.

If $e\in \R^{d+1} \setminus \R^d$ then we deduce, again from Lemma
$6.5$,  that for each $i = 1, 2, \ldots , d $
$$
\lim_{t\to 0} \frac{\partial \Phi}{\partial x_i}(te) =
\frac{\partial \phi}{\partial x_i} = 0
$$
so $\nabla \phi (0) = 0$. The same argument above, applied to any of
the coordinate directions, provides a contradiction as well.

\end{proof}

We are now ready to prove that, provided $\phi $ satisfies condition
$\mathcal H$, then the gradient of the Poisson extension of $f_{b,
\phi}$ defined by $(7.1)$ verifies a $1/N$-weak QR condition for
some integer $N \geq 2$. We need to recall the concept of almost
periodic function and some of their basic properties.

Given $g : \R^d \to \R$ and $\epsilon >0$, we say that $\tau \in
\R^d$ is an \textit{almost period} of $g$ relative to $\epsilon$ if
$$
\sup \{ |g(x +\tau ) - g(x) |\, \, : \, \, x\in \R^d \} \leq
\epsilon
$$
A continuous function $g: \R^d \to \R$ is said to be \textit{almost
periodic }(a.p.)  if for any $\epsilon
> 0$ there exists $l> 0$ such that any cube $Q\subset \R^d$ of
sidelength $l$ contains an almost period of $g$ relative to
$\epsilon$. A mapping  $g: \R^d \to \R^d$ is almost periodic if so
is each of its components. As in the classical case $d=1$, almost
periodic functions in $\R^d$  turn out to be those that can be
uniformly approximated by trigonometric polynomials. Any continous
function which is periodic in each variable is almost periodic in
$\R^d $. It can also be shown that finite sums and uniform limits of
almost periodic functions are almost periodic too. We refer to the
classical monograph by Besicovitch ([3]) for these and other
results.

Now let $\phi $, $f_{b, \phi}$ and  their respective Poisson
extensions $\Phi$ and $F$ be as in the beginning of the section.
Differentiating $(7.2)$ twice we get
$$
HF(x,y) = \sum_{n=0}^{\infty} b^n H\Phi (b^n x , b^n y)
$$
where the series converges uniformly on any strip $0< a \leq y \leq
b$ thanks to the exponential decay provided by Proposition $6.3$. In
our case of interest,  $0< a = \eta (\phi , b, d)< 1 = b $. Now for
$\eta \leq y \leq 1$ and  $e\in \R^{d+1}$, unitary,  consider the
mapping
$$
x\rightarrow (HF)e (x, y ) \, \, \, , \, \, \, (x \in \R^d )
$$
Since, for each $n\geq 0$, $H\Phi (b^n x , b^n y ) e$ is
$b^{-n}$-periodic in $x$, it follows from uniform convergence that
$(HF)e(x, y )$ is a.p. in $x$. But actually, a bit more is true:
from the inequality
$$
|(HF)e(x+\tau , y ) - (HF)e(x,y)| \leq ||HF(x+\tau , y ) - HF(x, y )
||
$$
and the basic properties of almost periodic functions it can be
shown that the almost periodicity of $x \to (HF)e(x, y)$ is uniform
in $e$ and $y\in [\eta  , 1]$ in the sense that, given $\epsilon
>0$, the $l$ in the definition of almost periodicity will only
depend on $\epsilon$, $\phi$, $b$ and $d$ but not on $e$.  This fact
will be useful in the proof of the following lemma.

For $0< \delta < 1$ and any cube $Q\subset \R^d$ we denote, as in
section $4$,  by $C_{\delta }(Q) = Q \times [\delta l(Q) , l(Q)]$
the $\delta$-Carleson box associated to $Q$.

\begin{lem}
Let  $\phi : \R^d \to \R$ be as in Proposition $6.3$. Assume in
addition   that $\phi$  satisfies condition $\mathcal H$. For $b
> 1$, let $f$ be the Weierstrass function associated to $b$
and $\phi$ as in $(7.1)$ and let $F$ be the Poisson extension of
$f$, given by $(7.2)$. Then there are positive constants $\delta =
\delta (  \phi , b  ) < 1$ and $c = c( \phi , b )$ such that for any
cube $Q\subset \R^d$ of sidelength $l(Q) \leq 1$
$$
\inf_{|e| =1} \, \sup_{(x,y) \in C_{\delta}(Q)}|(HF)e(x,y)| \geq
\frac{c}{y}
$$
where the infimum is taken over all unitary vectors $e\in \R^{d+1}$.
\end{lem}
\begin{proof}
From the functional equation $(7.5)$ for the Hessian we get
\begin{equation}
(HF )e (b^{-1}x, b^{-1}y ) = b (HF )e(x,y) + (H\Phi )e(b^{-1}x ,
b^{-1}y )
\end{equation}
Iterating $(7.7)$ we obtain:
$$
(HF)e(b^{-k}x, b^{-k}y ) = b^k (HF)e(x,y) + \sum_{n=0}^{k-1}b^n
(H\Phi )e(b^{n-k}x, b^{n-k}y )
$$
for any non negative integer $k=1,2,\ldots$. Therefore,
\begin{equation}
|(HF)e(b^{-k}x, b^{-k}y)| \geq b^k ( |(HF)e(x,y)| - \frac{||H\Phi
||}{b - 1} )
\end{equation}
From Proposition $7.2$ there exists $0< \eta < 1$ so that for any
unitary vector $e \in \R^{d+1}$ there exist $x_0 \in Q_0 = [-1/2 ,
1/2]^d$ and $\eta < y_0 < 1$ such that
\begin{equation}
|(HF)e(x_0 , y_0 )| > \frac{3||H\Phi ||_\infty}{b -1} \, .
\end{equation}
Now, from the previous remarks, the function $x \rightarrow (HF)e(x,
y_0 )$ is a.p , uniformly in $e$ and $y_0$ so we can choose $l =
l(\phi , b, d) \geq 1 $ such that for any cube $Q'\subset \R^d$ with
$l(Q' ) \geq l$ there is an almost period $\tau \in Q'$ relative to
$\epsilon = ||H\Phi ||_\infty / (b -1)$ that is uniform respect to
$e$. In other words
\begin{equation}
|(HF)e( x+\tau , y_0 ) - (HF)e(x, y_0 ) | \leq \frac{||H\Phi
||_\infty }{b-1}
\end{equation}
for any $x\in \R^d$ and any unitary $e\in \R^{d+1}$. Fix a non
negative integer $k$ such that $b^{-k} l <  l(Q) \leq b^{-k + 1} l$.
Let $Q' = b^k Q - x_0$ where , for $a> 0$, $aQ = \{ ax : \, x\in Q
\}$. Then $l(Q') \geq l$ so there is $\tau \in Q'$ satisfying
$(7.10)$. Observe that $b^{-k}(x_0 + \tau ) \in Q$. From
$(7.8)$,$(7.9)$ and $(7.10)$ we get
$$
b^{-k}y_0 |(HF)e(b^{-k}(x_0 + \tau , y_0 )) | \geq y_0 \Big [
|(HF)e(x_0 + \tau , y_0 )| - \frac{||H\Phi ||_\infty}{b -1} \Big
]\geq \eta \frac{||H\Phi ||_\infty}{b -1}
$$
On the other hand
$$
\frac{\eta}{b l}l(Q)< b^{-k}y_0 < l(Q)
$$
so the conclusion follows taking $\delta = \eta / b l $ and $c= \eta
||H\Phi ||_\infty / (b -1)$.
\end{proof}

\begin{cor}
Let $b$, $\phi$, $f$, $F$ be as in Lemma $7.3$. Then there are
 $N = N ( \phi, b, d) \in \N$, $N \geq 2$ and $\gamma = \gamma
(\phi ,b, d ) \geq 1$ such that $\nabla F$ verifies a  $1/N$-weak QR
condition with constant $\gamma$, that is, for any cube $Q\subset
\R^d$
$$
\int_{C_{1/N}(Q)}\max_{|e|=1}\, |HF(x,y)|^2 dxdy \leq \gamma^2
\min_{|e|=1}\int_{C_{1/N}(Q)} |(HF)e(x,y)|^2 dxdy
$$
\end{cor}
\begin{proof}
Given a cube $Q$ in $\R^d$, let $ \widetilde{Q}$ be the cube with
the same center and half its sidelength.   From Lemma $7.3$, there
are $0 < \delta = \delta (\lambda , \phi ) < 1$, $c = c( \phi , b )
> 0$ such that for any unitary vector $e \in \R^{d+1}$ there exists
a point $(x,y) \in C_{\delta}(\widetilde{Q})$ such that
$$
|(HF)e(x,y)| \geq \frac{2c}{l(Q)}
$$
Let $B \subset \R^{d+1}_+ $ be the ball centered at $(x,y)$ of
radius $\delta l(Q) / 4$. Then from subharmonicity,
$$
|(HF)e(x,y)|^2 \leq  \frac{C_1}{l(Q)^{d+1}} \int_{B} |(HF)e|^2
$$
for some $C_1 = C_1 (d)$. Choose $N\in \N $ such that $\frac{1}{N}
\leq \frac{\delta}{4} \leq \frac{1}{N-1}$. Then  $B\subset
C_{1/N}(Q)$ and, combining the two previous inequalities we get
\begin{equation}
C_2 (l(Q))^{d-1} \leq \min _{|e|=1} \, \int_{C_{1/N}(Q)} |(HF)e|^2
\end{equation}
for some constant $C_2 = C_2 (b , d, \phi )$. On the other hand,
from Proposition $7.1$,
$$
\sup_{C_{1/N}(Q)} ||HF||^2 \leq \frac{C_3}{(l(Q))^2}
$$
where $C_3 = C_3(b , d, \phi )$. In particular
\begin{equation}
\int_{C_{1/N}(Q)} ||HF||^2 \leq C_4 (l(Q))^{d-1}
\end{equation}
The conclusion follows from $(7.11)$ and $(7.12)$.

\end{proof}

\section{Proof of Theorem $3$ }

% Theorem $3$ is the analogue in this context of  Corollary $1.3$ in
% [9]. We will omit the proof, since it follows the same argument than
% in [9],  together with the following remark (see [10]): if $f: \R^d
% \to \R $ is bounded, continous, and
% $$
% |f(x + h) +f(x-h) -2f(x)| \leq C |h|^2
% $$
% for each $x$, $h\in \R^d$ and some constant $C > 0$ then $f\in C^{1,
% \alpha} (\R^d )$ for any $0< \alpha < 1$. ( Actually, a little bit
% more can be said:  the modulus of continuity of $\nabla f$ is $O (
% \delta \log (\frac{1}{\delta}))$ ).

% \

We prove now Theorem $3$.

\begin{proof}

% The fact that  $f_{b, \phi} \in \Lambda_{*}(\R^d )$ uses an standard
% trick in the theory of lacunary series. Fix $x, h\in \R^d$ and
% choose an integer $N$ such that $b^{-(N+1)} < |h| \leq b^{-N}$. We
% use the notation $\Delta_2 f (x,h) = f(x+h) + f(x-h) - 2 f(x)$.
% Then, split the double difference as
% $$
% \Delta_2 f_{b, \phi}(x,h) = \sum_{n=0}^N b^{-n} \Delta_2 \phi (b^n
% x, b^n h) +  \sum_{n= N+1}^{\infty} b^{-n} \Delta_2 \phi (b^n x, b^n
% h)
% $$
% In the fist term we use that $|\Delta_2 \phi (b^n x, b^n h)| <
% C||H\phi ||_\infty |b^n h|^2$ and in the second that $|\Delta_2 \phi
% (b^n x, b^n h)| < 4 ||\phi ||_{\infty} $. Then we deduce
% $$
% | \Delta_2 f_{b, \phi}(x,h) | <  C||H\phi ||_\infty |h|^2
% \sum_{n=0}^N
%  b^{n}  +  4 ||\phi ||_{\infty} \sum_{N+1}^{\infty} b^{-n}
% $$
% The choice of $N$ gives that $f_{b, \phi} \in \Lambda_{*}(\R^d )$.

% Fix $x\in \R$, and  $h\in \R$, $h\neq 0$. Assume that $b^{-(N+1)} <
% |h| \leq b^{-N}$. into two parts: from $0$ to $N$ and from $N+1$ to
% $\infty$.the second derivative of the cosine and in the second block
% we just replace the cosine by $1$. Then we get:
% $$
% |f_{b,1}(x+h) + f_{b,1}(x+h)-2f_{b,1}(x)| \leq \sum_{n=0}^N
% b^{-n}(b^n h )^2 + 4 \sum_{N+1}^{\infty} b^{-n} \leq C(b)|h|
% $$
% so $f_{b,1}\in \Lambda_{*}(\R )$.

The fact that  $f_{b, \phi} \in \Lambda_{*}(\R^d )$ follows  from
Proposition $7.1$. The fact that $f_{b, \phi}$ is nowhere
differentiable follows from Theorem $3.1$ in [10]. Pick  $x_0 \in
\R^d$. If $x_0 \neq 0$, define $e = x_0 / |x_0 |$ and observe that
the one-variable function $t\to f(te)$ satisfies the hypothesis of
Theorem $3.1$ in [10] therefore is nowhere differentiable.  In
particular $f$ has no directional derivative at $x_0$ along the
radial direction $e$. If $x_0 = 0$ the same argument shows that $f$
has no directional derivative at $0$ along any direction. This
proves part $1$.

Let $0 < \delta <1 $ and $c>0$ be the constants appearing in  Lemma
$7.3$. Fix $x_0 \in \R^d$ and a unitary vector $e \in \R^d$. Let
$Q_k$ be the cube centered at $x_0$ and sidelength $\delta^{k}$.
 From Lemma $7.3$, there
is $(x,y)\in Q_k \times [\delta^{k+1}, \delta^k ]$ such that
$\displaystyle | HF \cdot e (x,y)| \geq c / y $. Let $B_k$ be the
ball centered at $(x,y)$ of radius $\delta^k / 2$. From
subharmonicity,
\begin{equation}
\frac{C}{\delta^{2k}} \leq \frac{1}{\delta^{(k+1)(d+1)}}\int_{B_k}
\big |\nabla ( D_e F ) \big |^2 dx dy
\end{equation}
where $D_e F$ means the derivative of $F$ in the direction $e$ and
$C= C(d, b, \phi)$. It is easy to check that
\begin{equation}
\displaystyle B_k \subset Q'_k \times [\frac{\delta^{k+1}}{2},
\delta^k - \frac{\delta^{k+1}}{2}] \subset \Gamma (x_0 )
\end{equation}
where
$$
\Gamma (x_0 ) = \{(x, y) \in \R^{d+1}_+ \, : |x-x_0 | \leq ( 1 +
\frac{\sqrt d}{\delta} ) y \}
$$
On the other hand, from $(8.1)$ and $(8.2)$ we get
\begin{equation}
\int_B  y^{1-d} \, |\nabla ( D_e F ) (x,y) |^2 dx dy \geq C >0
\end{equation}
for some $C>0$ independent of $k$.  Since each cone $\Gamma (x_0)$
contains infinitely many disjoint balls $B_k$, we get from $(8.3)$
that
$$
\int_{\Gamma (x_0 )} y^{1-d}\, \big |\nabla ( D_e F) (x,y) \big |^2
dx dy = \infty
$$
which implies that the Area function of $D_e F$ for cones of some
fixed aperture is infinite for all $x_0 \in \R^d$. From the Area
version of the Local Fatou Theorem for harmonic functions ( [16,
Chap. VII, Thms. $3$, $4$]), it follows that $D_e F$ is
non-tangentially unbounded at almost every point $x\in \R^d$.
Theorem $3$, part $2$ follows now from Proposition $2.1$ and
Corollary $2.4$.

\

To prove part $3$ of Theorem $3$, let $b$, $\phi$ and $f_{b, \phi}$
as in the statement of the theorem and let $F$ be the Poisson
extension of $f_{b, \phi}$. Let $N\in \N$ and $\gamma \geq 1$ be as
in Corollary $7.4$. The result follows from corollaries $3.8$ and
$4.5$.

\end{proof}

\section{Remarks and questions}

\begin{enumerate}

\item Quasi-regular mappings can be understood as a sort of higher dimensional analogue
 of holomorphic mappings. On the other hand, in harmonic analysis, the higher dimensional analogue
  of holomorphic mappings are the gradients of harmonic functions, or equivalently, systems
  of conjugate harmonic functions (see [16, p. 65]). Can one describe the harmonic functions
  in an upper half-space whose gradient is (weakly) quasi-regular?
   Quasi-regularity is typically an involved property to handle
with. If $u$ is harmonic in $\R^d$ , for $d\geq 3$ and $u$ is
independent of at least one direction, then it is clear that $\nabla
u$ cannot be quasi-regular. To which extent is this the only
obstacle to the quasi-regularity of a harmonic gradient?

\item The weak quasi-regularity condition used in section
$4$ raises some natural but subtle questions. Suppose that $\nabla F
$ is a harmonic gradient. It is not clear (and probably false) that
if $\nabla F$ satisfies a $\delta$- weak QR condition then it also
satisfies a $\delta'$-weak QR condition for $\delta' \simeq \delta$.
The way to show that $\nabla F$ satisfies a weak QR condition
(Corollary $7.4$) relies on the functional equation $(7.5)$ and on a
certain
 lower uniform bound for the
Hessian $HF$ (Lemma $7.3)$ together with sub-harmonicity. This is a
sort of by-pass that avoids the problem of comparing directly the
maximal and the minimal distortions of $\nabla F$. Because of this,
we have been unable to adapt the method in part $3$ of Theorem $3$,
to cover the case of lacunary series in $\lambda_{*}$ of the form
$$
f(x) = \sum_{n=0}^{\infty} {\varepsilon}_n b^{-n} \, \phi (b^n x)
$$
where $\{{\varepsilon}_n \}$ is a sequence of real numbers tending
to $0$.

\item The dichotomy result given by Proposition $5.1$ can
dramatically fail if either harmonicity or quasi-regularity are
dropped from the hypothesis. Indeed, if $u$ is harmonic in $\R^d$,
$d\geq 2$ and $u$ does not depend on one of the variables then
$\nabla u$ is not quasi-regular (even in the weak sense) and the
conclusions of Proposition $5.1$ obviously do not hold. On the other
hand, there is a bounded quasi-regular mapping $g: \R^2_+ \to
\mathbb{C}$ that fails to have vertical limit al almost all $x\in
\R$, as the following construction shows. Let $h$ be an increasing,
singular, quasi-symmetric homeomorphism of $\R$ into $\R$ ( see [4],
Theorem $3$). This implies the existence of $E \subset \R $ such
that $m_1 (\R \setminus E) = m_1 (h(E)) = 0$. Extend $h$ to a
quasi-conformal map $H: \R^2_+ \to \R^2_+ $ ([3], Theorem $1$). Now
take a  bounded analytic function $f: \R^2_+ \to \mathbb{C} $ such
that for any $x\in h(E)$, $f$ fails to have limit along any curve
ending at $x$ ([6], Lemma $1$, Ch.$2$). The statement follows by
taking  $g = f\circ H $ (see [2], Theorem $5.5.1$ for the
quasi-regularity of $g$).

\item Even if $d=1$, the authors wonder which part of the results in [10] can be saved
if the base function $\phi$ is only assumed to be Lipschitz. On the
other hand it is also natural to ask under what extent the
periodicity or almost periodicity of $\phi$ is essential for the
nowhere differentiability of the Weierstrass function.

\end{enumerate}

\end{document}